\documentclass[11pt,reqno]{amsart}
\usepackage[T1]{fontenc}
\usepackage{lmodern}
\usepackage{tikz}
\usepackage{comment}
\usepackage{amsmath,amssymb,mathtools,mathrsfs}
\usepackage{microtype}
\usepackage{needspace}
\usepackage[a4paper,textwidth=163mm,textheight=238mm,centering]{geometry}
\usepackage{xcolor}
\usepackage[
colorlinks=true,urlcolor=blue,linkcolor=blue,citecolor=magenta]{hyperref}
\usepackage[lite,initials]{amsrefs}
\numberwithin{equation}{section}
\newtheorem{theorem}{Theorem}[section]
\newtheorem{lemma}[theorem]{Lemma}
\newtheorem{proposition}[theorem]{Proposition}
\newtheorem{corollary}[theorem]{Corollary}
\theoremstyle{remark}
\newtheorem{remark}[theorem]{Remark}
\allowdisplaybreaks[2]
\hypersetup{pdftitle={Curvature quotient flows and Alexandrov--Fenchel inequalities in hyperbolic space}}
\title[Alexandrov--Fenchel inequalities in hyperbolic space]{Curvature quotient flows and Alexandrov--Fenchel inequalities in hyperbolic space}

\author[T. Luo]{Tianci Luo}
\author[Y. Wei]{Yong Wei}
\address{School of Mathematical Sciences, University of Science and Technology of China, Hefei 230026, P.R. China}
\email{Luo\_tianci@mail.ustc.edu.cn}
\email{yongwei@ustc.edu.cn}

\author[R. Zhang]{Ruijia Zhang}
\address{Department of Mathematics, Sun Yat-sen University, Guangzhou, 510275, China}
\email{zhangrj76@mail.sysu.edu.cn}

\author[R. Zhou]{Rong Zhou}
\address{School of Mathematical Sciences, University of Science and Technology of China, Hefei 230026, P.R. China}
\email{zhourong@mail.ustc.edu.cn}

\date{\today}
\begin{document}

\begin{abstract}
We study quermassintegral-preserving curvature flows in
$\mathbb H^{n+1}$, $n\geq3$, driven by the curvature quotient
$(E_k/E_\ell)^{1/(k-\ell)}$. For $1\leq\ell<k$ with $k+\ell\leq n$, we prove that
positive sectional curvature is preserved.
Starting from a smooth closed strictly convex hypersurface
with positive sectional curvature, the flow exists for
all time and converges smoothly and exponentially
to a geodesic sphere.
As an application, we establish sharp Alexandrov--Fenchel
inequalities between $W_k$ and $W_\ell$ for smooth bounded
convex domains whose boundaries have nonnegative sectional
curvature, whenever $0\leq\ell<k\leq n$ and $k+\ell\leq n$.
Equality holds if and only if the domain is a geodesic ball.
Together with previously known cases, these inequalities
give the full family of Alexandrov--Fenchel inequalities
for this class of domains in $\mathbb H^6$.
We also construct counterexamples showing that
$k+\ell\leq n$ is the optimal range for preservation
of positive sectional curvature by these flows.
\end{abstract}
\maketitle
\tableofcontents

\section{Introduction}\label{sec-intro}
The Alexandrov--Fenchel inequalities are fundamental inequalities in
convex geometry. In Euclidean space, they give sharp comparisons
between the quermassintegrals of a convex body, with equality precisely
for balls. We refer to Schneider \cite{Schneider} for the classical
theory. Curvature flows provide an important approach to these
inequalities by deforming a hypersurface towards a sphere while
preserving one geometric quantity and decreasing another. Important
examples include the volume-preserving mean curvature flow of Huisken
\cite{Huisken}, the mixed-volume-preserving curvature flows of McCoy
\cite{McCoy}, and the inverse curvature flow of Guan and Li
\cite{GuanLi2009}. The latter also gives the quermassintegral
inequalities for suitable $k$-convex star-shaped domains in Euclidean
space.

For convex domains in hyperbolic space, the quermassintegrals are
naturally defined by integral geometry. Unlike their Euclidean
counterparts, their curvature integral expressions contain additional
lower order terms arising from the ambient curvature. We refer to
Santal\'o \cite{Santalo} and Solanes \cite{Solanes2006} for the
corresponding integral geometric formulas.

Let $\mathbb H^{n+1}$ denote the simply connected space form of constant
sectional curvature $-1$. Let $\Omega\subset\mathbb H^{n+1}$ be a smooth
bounded convex domain with boundary $M=\partial\Omega$, and let
$\kappa=(\kappa_1,\ldots,\kappa_n)$ be its principal curvatures with
respect to the outward unit normal $\nu$. We write
\begin{equation*}
 E_j(\kappa)=\binom nj^{-1}\sigma_j(\kappa),
 \qquad
 \sigma_j(\kappa)=\sum_{i_1<\cdots<i_j}
                  \kappa_{i_1}\cdots\kappa_{i_j},
 \qquad E_0=1.
\end{equation*}
Let $W_j$, $0\leq j\leq n$, denote the hyperbolic quermassintegrals
in the normalization of Andrews, Chen and Wei \cite{ACW}.
Their definition and first variation are recalled in
Section~\ref{subsec-quermassintegrals}.
For a geodesic ball $B_r$ of radius $r>0$, set
\begin{equation*}
 f_j(r)=W_j(B_r),\qquad 0\leq j\leq n.
\end{equation*}
These functions are strictly increasing. The hyperbolic
Alexandrov--Fenchel problem is to establish the sharp comparison
\begin{equation}\label{intro-AF-problem}
 W_j(\Omega)\geq f_j\circ f_\ell^{-1}\bigl(W_\ell(\Omega)\bigr),
 \qquad 0\leq \ell<j\leq n,
\end{equation}
with equality precisely for geodesic balls. Different positive
normalizations of the individual quermassintegrals do not change this
comparison. In particular, the normalization used in \cite{HL} is
$(n+1-j)/(n+1)$ times our $W_j$.

A smooth convex hypersurface is called horospherically convex if
$\kappa_i\geq1$ for every $i$. Under this assumption, Cabezas-Rivas and
Miquel \cite{CM} proved long-time existence and convergence of the
volume-preserving mean curvature flow. Wang and Xia \cite{WX}
established the full family of inequalities
\eqref{intro-AF-problem} by quermassintegral-preserving curvature
flows. Andrews and Wei \cite{AW} subsequently proved convergence for a
broad class of such flows. Hu, Li and Wei \cite{HLW} gave another proof
of the inequalities \eqref{intro-AF-problem} using locally constrained curvature flows.
Related sharp inequalities involving curvature integrals and
quermassintegrals were obtained by Ge, Wang and Wu \cite{GeWangWu}.

Several cases are known under weaker geometric assumptions.
Brendle, Guan and Li \cite{BGL} proved the comparison between $W_2$ and
$W_1$ for star-shaped mean-convex domains. See also
\cites{GuanLiSurvey,HLW}. Li, Wei and Xiong \cite{LWX} established the
comparison between $W_3$ and $W_1$ for star-shaped $2$-convex domains.
Hu and Li \cite{HLSectional} extended the odd-index quermassintegral
comparisons with $W_1$ obtained in \cite{GeWangWu} to hypersurfaces with
nonnegative intrinsic sectional curvature. 
At the other end of the index range, Brendle, Guan and Li
\cite{BGL} proved \eqref{intro-AF-problem} for smooth
strictly convex domains with $0\leq \ell<j=n$.
Andrews, Hu and Li \cite{AHL2020} used the harmonic mean
curvature flow to establish several comparisons with upper
index $n-1$ under strict convexity. Subsequently, Hu and Li
\cite{HL} established all comparisons with upper index
$n-1$ under the same assumption, using the duality between
hyperbolic and de Sitter space. Luo and Zhou \cite{LZ} gave another proof of the comparison
between $W_{n-1}$ and $W_0$ for smooth convex domains,
using a locally constrained Gauss curvature flow.

Of particular relevance to the present work is the volume-preserving
flow of Andrews, Chen and Wei \cite{ACW}. They proved long-time
existence and exponential convergence to a geodesic sphere for a class
of speeds including positive powers of the normalized mean curvatures,
assuming that the initial hypersurface has positive intrinsic
sectional curvature. For a convex boundary, the Gauss equation
identifies this assumption with
\begin{equation}\label{intro-positive-section}
 \kappa_i\kappa_j>1\qquad\text{for all }i\ne j.
\end{equation}
This condition does not require each principal curvature to be at
least one. Their convergence theorem yields
\begin{equation}\label{ACW-volume}
 W_j(\Omega)\geq f_j\circ f_0^{-1}\bigl(W_0(\Omega)\bigr),
 \qquad 1\leq j\leq n,
\end{equation}
for domains whose boundaries have positive sectional curvature, with equality precisely for geodesic balls. 
For the volume-preserving Gauss curvature flow, Wei, Yang
and Zhou \cite{WYZ2024} subsequently proved long-time
existence and smooth convergence for every positive power
of the Gauss curvature under strict convexity alone.
Their result also gives the comparison between $W_n$ and
$W_0$ for smooth weakly convex domains.


In this paper we consider $W_\ell$-preserving quotient flows with
$\ell\geq1$ and use them to obtain comparisons between
positive-index quermassintegrals under the sectional
curvature condition considered in \cite{ACW}.

\subsection{Main results}\label{subsec-main}
Our main geometric consequence is the following family of
Alexandrov--Fenchel inequalities.

\begin{theorem}\label{AF}
Let
$\Omega\subset\mathbb H^{n+1}$ be a smooth bounded convex domain
whose boundary has nonnegative sectional curvature, where $n\geq3$.
For every pair of integers $0\leq \ell<k\leq n$ with $k+\ell\leq n$,
\begin{equation}\label{AF-all}
 W_k(\Omega)\geq f_k\circ f_\ell^{-1}\bigl(W_\ell(\Omega)\bigr).
\end{equation}
Equality holds if and only if $\Omega$ is a geodesic ball.
\end{theorem}

The comparisons with $\ell=0$ under positive sectional
curvature were proved by Andrews, Chen and Wei \cite{ACW}. We include $\ell=0$ in Theorem~\ref{AF} to state the inequalities uniformly
under nonnegative sectional curvature.
For $k\leq\lfloor(n+1)/2\rfloor$, Theorem~\ref{AF}
gives all comparisons with $0\leq\ell<k$.
In particular, when $n=5$, it includes the comparison
between $W_3$ and $W_2$. Together with the inequalities
for upper index $n-1$ and $n$, proved by Hu and Li
\cite{HL} and Brendle, Guan and Li \cite{BGL},
respectively, we obtain the following result.

\begin{corollary}\label{AF-H6}
Let $2\leq n\leq5$, and let
$\Omega\subset\mathbb H^{n+1}$ be a smooth bounded convex domain
whose boundary has nonnegative sectional curvature. Then
\begin{equation}\label{AF-low-dimensional-all}
 W_k(\Omega)\geq
 f_k\circ f_\ell^{-1}\bigl(W_\ell(\Omega)\bigr),
 \qquad 0\leq \ell<k\leq n.
\end{equation}
Equality holds if and only if $\Omega$ is a geodesic ball.
\end{corollary}

Fix integers $1\leq \ell<k\leq n-1$ with $k+\ell\leq n$.
Starting from a smooth embedding $X_0$ of a closed manifold $M^n$
into $\mathbb H^{n+1}$, we consider
\begin{equation}\label{flow}
\begin{cases}
 \displaystyle
 \frac{\partial X}{\partial t}(x,t)
 =\bigl(\phi(t)-F(\mathcal W)\bigr)\nu(x,t),\\
 X(\cdot,0)=X_0,
\end{cases}
\end{equation}
where
\begin{equation}\label{normalization}
 F=\left(\frac{E_k}{E_\ell}\right)^{1/(k-\ell)},
 \qquad
 \phi(t)=
 \frac{\displaystyle\int_{M_t}E_\ell F\,d\mu_t}
      {\displaystyle\int_{M_t}E_\ell\,d\mu_t}.
\end{equation}
Here $M_t=X(M,t)$, $\mathcal W=(h^i{}_j)$ is its Weingarten
matrix, $\nu$ is the outward unit normal, and $d\mu_t$ is the
induced area element.

This flow belongs to the quermassintegral-preserving quotient family
considered in \cite{WX} for horospherically convex hypersurfaces.
It has the additive form studied in \cite{ACW}, with the global
term chosen to preserve $W_\ell$. 
The normalization preserves $W_\ell$ and makes $W_k$ nonincreasing. By \eqref{W-monotonicity}, $W_k$ is constant precisely when $F$
is spatially constant. We prove the following long-time existence and convergence theorem.

\begin{theorem}\label{main-flow}
Let $n\geq3$ and $1\leq \ell<k\leq n-1$ satisfy $k+\ell\leq n$.
Suppose that $X_0$ is a smooth embedding of a closed manifold $M^n$
into $\mathbb H^{n+1}$, and that $M_0=X_0(M)$ bounds a convex
domain $\Omega_0$ and has positive sectional curvature. Then
\eqref{flow}--\eqref{normalization} has a unique smooth solution
for all $t\in[0,\infty)$. Each $M_t$ has positive sectional
curvature, and $M_t$ converges smoothly and exponentially to a
geodesic sphere $\partial B_{r_\infty}(p_\infty)$, where $W_\ell(B_{r_\infty})=W_\ell(\Omega_0)$.
\end{theorem}

Since $W_\ell$ is preserved and $W_k$ is nonincreasing,
Theorem~\ref{main-flow} proves \eqref{AF-all} for
$1\leq\ell<k$ with $k+\ell\leq n$ when the boundary has positive sectional
curvature. The case $\ell=0$ follows from \cite{ACW}.
Klein contraction and a short-time flow argument extend
the result, including the equality case, to nonnegative
sectional curvature, as shown in Section~\ref{subsec-convergence}.

The flow \eqref{flow}--\eqref{normalization} is also defined for
$1\leq \ell<k\leq n$ outside the preceding range, including $k=n$.
The next theorem shows that $k+\ell\leq n$ is necessary for
preservation of positive sectional curvature.

\begin{theorem}\label{prop-section-counterexample}
Let $n\geq3$ and $1\leq \ell<k\leq n$ satisfy $k+\ell>n$. There are smooth closed strictly convex initial hypersurfaces with positive sectional curvature whose solutions of~\eqref{flow}--\eqref{normalization} develop negative sectional curvature while remaining smooth and strictly convex. The loss can occur in arbitrarily short time.
\end{theorem}

Together with Theorem~\ref{main-flow}, this gives the exact range
$k+\ell\leq n$ for preservation of positive sectional curvature by
these quotient flows. For adjacent quotients $F=E_{m+1}/E_m$,
the preservation range is $2m<n$, and the counterexamples cover
$2m\geq n$.

\subsection{Discussion of the proof}

We first prove preservation of positive sectional curvature
for the quotient flow. Following \cite{ACW}, we consider
smooth lower supports of the minimum sectional curvature
$S=\kappa_1\kappa_2-1$. The usual concavity and
inverse-concavity of $F$ do not by themselves control all
the derivative terms in the evolution of $S$.
At a contact point where a lower support has zero
spatial gradient, the Codazzi equation gives
\begin{equation*}
 \kappa_2\nabla_m h_{11}
 +\kappa_1\nabla_m h_{22}=0.
\end{equation*}
This relation allows us to apply the paired concavity
inequality in Proposition~\ref{key-lemma} to the dual curvature
quotient. The inequality holds precisely when
$k+\ell\leq n$ and gives preservation of positive
sectional curvature in this range. The argument applies
to any positive smooth global term $\phi(t)$.

The main additional difficulty for the $W_\ell$-preserving
flow is to obtain a uniform lower bound for the inradius.
Alexandrov reflection gives a uniform diameter bound,
as in \cite{ACW}. Unlike the volume-preserving case,
a positive lower bound for $W_\ell$ does not by itself
prevent collapse when $\ell\geq1$.
In Lemma~\ref{noncollapsing}, we use the sectional
curvature condition and a three-dimensional section
of a John ellipsoid to show that, if its shortest
semiaxis tends to zero, then so does its second
largest semiaxis. Thus the domains can only collapse
to sets of dimension at most one, which would force
$W_\ell$ to tend to zero. This gives the required
inradius bound.

With the inner and outer radius bounds, a Tso-type
estimate \cite{Tso} gives an upper bound for $F$.
The sectional curvature condition then gives a
positive lower bound for the principal curvatures.
The term $-\phi\mathcal W^2$ in the Weingarten
evolution gives the upper curvature bound.
These estimates yield long-time existence and
uniform higher-order bounds.

The decrease of $W_k$ gives a sequence of times
along which the dissipation tends to zero.
A smooth subsequential limit has constant
$E_k/E_\ell$ and is therefore a geodesic sphere
by Koh's rigidity theorem \cite{Koh}.
The linearization and normal stability theorem
\cite{PSZ} then give smooth exponential convergence
of the entire flow. For the Alexandrov--Fenchel
inequalities, Klein contraction allows us to pass
from positive to nonnegative sectional curvature.
The equality cases follow by running the
constrained flow for a short time and using
the monotonicity of $W_k$.

Finally, we show that the index restriction is
necessary for curvature preservation by this family
of flows. When $k+\ell>n$, we construct a local
graph whose minimum sectional curvature is
nonnegative, vanishes at one point, and has
negative initial time derivative there.
The global term does not contribute to this
initial derivative at the point where $S=0$.
A quartic correction makes the sectional curvature
positive away from that point. We then join the
graph to a spherical cap and use Klein contraction
to obtain smooth closed initial hypersurfaces
with positive sectional curvature.
By short-time continuous dependence, their
flows lose this property in arbitrarily short time.

\subsection{Organization of the paper}
Section~\ref{sec-prelim} collects the curvature-function properties,
hyperbolic geometry and Klein contraction, quermassintegral
variations, and evolution equations used below.
Section~\ref{sec-preservation} establishes preservation of positive
sectional curvature. In Section~\ref{sec-proof}, we prove the shape
and curvature estimates, long-time existence, convergence, and the
Alexandrov--Fenchel inequalities. Section~\ref{sec-counterexamples}
proves sharpness of the index range by geometric counterexamples.
The proof of the paired concavity inequality is given in
Section~\ref{sec-algebra}.

\subsection{On the use of AI}  
The main mathematical ideas and geometric strategy of this
paper were developed by the authors. GPT-5.6 Sol and GPT-6
Astra were used to assist with calculations and proof
development, particularly the algebraic inequality in
Section~6, as well as with improving the exposition and
correcting grammatical errors. The authors take full
responsibility for the mathematical content and conclusions
of this paper.

\section{Preliminaries}\label{sec-prelim}

In this section, we recall the properties of elementary symmetric
functions and the geometry of hypersurfaces in hyperbolic space.
We also record the Klein contraction, the variational formulas for
quermassintegrals, and the evolution equations used below.
Repeated tensor indices are summed from $1$ to $n$ unless otherwise stated.

\subsection{Elementary symmetric functions}\label{subsec-symmetric}

We use the normalized elementary symmetric functions
\begin{equation*}
 E_m(\kappa)=\binom nm^{-1}\sigma_m(\kappa),\qquad 0\leq m\leq n.
\end{equation*}
We set $E_0=1$ and $E_m=0$ for $m<0$ or $m>n$, and work
on the positive cone
\begin{equation*}
 \Gamma_+=\{\kappa\in\mathbb R^n:\kappa_i>0\text{ for all }i\}.
\end{equation*}
A smooth symmetric function $\Phi$ on $\Gamma_+$ is also regarded
as a function of positive definite symmetric matrices by evaluation
at their eigenvalues. For symmetric matrices $A$ and $\xi$, write
\begin{equation*}
 \left.\frac{d}{ds}\Phi(A+s\xi)\right|_{s=0}
 =\dot\Phi^{ij}\xi_{ij},\qquad
 \left.\frac{d^2}{ds^2}\Phi(A+s\xi)\right|_{s=0}
 =\ddot\Phi^{ij,k\ell}\xi_{ij}\xi_{k\ell}.
\end{equation*}
For derivatives in the eigenvalue variables, we write
$\dot\Phi^i=\partial\Phi/\partial\kappa_i$ and
$\ddot\Phi^{ij}=\partial^2\Phi/\partial\kappa_i\partial\kappa_j$.
At a diagonal matrix, the standard matrix derivative formula
\cite{Andrews2007}*{Theorem~5.1} gives
\begin{equation}\label{matrix-Hessian}
\begin{aligned}
 \dot\Phi^{ij}&=\dot\Phi^i\delta_{ij},\\
 \ddot\Phi^{ij,k\ell}\xi_{ij}\xi_{k\ell}
 &=\sum_{i,j}\ddot\Phi^{ij}\xi_{ii}\xi_{jj}
   +2\sum_{i<j}
       \frac{\dot\Phi^i-\dot\Phi^j}{\kappa_i-\kappa_j}\xi_{ij}^2.
\end{aligned}
\end{equation}
The divided differences extend continuously to repeated eigenvalues,
where their value is $\ddot\Phi^{ii}-\ddot\Phi^{ij}$.

On $\Gamma_+$, the Newton--MacLaurin inequalities
\cite{HLW}*{Lemma~2.2} take the form
\begin{equation}\label{Newton-MacLaurin}
 E_{m+1}E_{j-1}\leq E_mE_j,
 \qquad 1\leq j\leq m\leq n-1,
\end{equation}
with equality if and only if all principal curvatures are equal.

For $0\leq \ell<k\leq n$, set
\begin{equation*}
 F=\left(\frac{E_k}{E_\ell}\right)^{1/(k-\ell)}.
\end{equation*}
The adjacent quotients are strictly increasing and concave on
$\Gamma_+$ by \cite{Andrews2007}*{Theorem~2.6}. Since the geometric
mean is increasing and concave, the representation
\begin{equation}\label{quotient-product}
 F=\left(\prod_{j=\ell}^{k-1}\frac{E_{j+1}}{E_j}\right)^{1/(k-\ell)}
\end{equation}
shows that $F$ is smooth, positive, strictly increasing, homogeneous
of degree one, and concave. Moreover,
\begin{equation*}
 E_m(A^{-1})=\frac{E_{n-m}(A)}{E_n(A)},\qquad 0\leq m\leq n,
\end{equation*}
and hence
\begin{equation}\label{dual-quotient}
 F_*(A):=\frac1{F(A^{-1})}
 =\left(\frac{E_{n-\ell}(A)}{E_{n-k}(A)}\right)^{1/(k-\ell)}.
\end{equation}
The same argument with the complementary indices shows that
$F_*$ is concave. Thus $F$ is inverse-concave.

Homogeneity, monotonicity, and concavity give
\begin{equation}\label{quotient-homogeneity}
 \dot F^{ij}A_{ij}=F,\qquad F(\operatorname{Id})=1,\qquad
 \kappa_{\min}\leq F\leq\kappa_{\max},
\end{equation}
and
\begin{equation}\label{trace-bounds}
 \dot F^{ij}\delta_{ij}\geq1,\qquad
 \dot F^{ij}(A^2)_{ij}\geq F^2.
\end{equation}
The first inequality follows from concavity and
Euler's identity:
\begin{equation*}
1=F(\operatorname{Id})
\leq F(A)+\dot F^{ij}(\delta_{ij}-A_{ij})
=\dot F^{ij}\delta_{ij}.
\end{equation*}
Applying the same inequality to $F_*$ at $A^{-1}$
and using
$\dot F_*^i(\kappa^{-1})=\kappa_i^2\dot F^i/F^2$
gives the second inequality. Differentiating $\log F$ and using the identities in
\cite{HLW}*{Lemma~2.1}, we obtain
\begin{equation}\label{quotient-trace-identities}
\begin{aligned}
 \dot F^{ij}\delta_{ij}
 &=\frac{F}{k-\ell}
   \left(k\frac{E_{k-1}}{E_k}-\ell\frac{E_{\ell-1}}{E_\ell}\right),\\
 \dot F^{ij}(A^2)_{ij}
 &=\frac{F}{k-\ell}
   \left((n-\ell)\frac{E_{\ell+1}}{E_\ell}-(n-k)\frac{E_{k+1}}{E_k}\right).
\end{aligned}
\end{equation}

For Section~\ref{sec-preservation}, put
\begin{equation}\label{psi-definition}
 \psi=-F^{-1},\qquad -\psi(A^{-1})=F_*(A).
\end{equation}
Then $\psi$ is strictly increasing and concave. The symmetry and
concavity of $\psi$ and $F_*$ imply
\begin{equation}\label{psi-properties}
\begin{aligned}
 &\dot\psi^i\geq\dot\psi^j>0
 &&\text{if }\kappa_i\leq\kappa_j,\\
 &\kappa_i^2\dot\psi^i\leq\kappa_j^2\dot\psi^j
 &&\text{if }\kappa_i\leq\kappa_j,\\
 &\frac{\dot\psi^i-\dot\psi^j}{\kappa_i-\kappa_j}
       +\frac{\dot\psi^i}{\kappa_j}
       +\frac{\dot\psi^j}{\kappa_i}\geq0
 &&\text{if }i\ne j.
\end{aligned}
\end{equation}
The last inequality follows from the second and
$\dot F_*^i(\kappa^{-1})=\kappa_i^2\dot\psi^i$,
using
\begin{equation*}
 \frac{\dot\psi^i-\dot\psi^j}{\kappa_i-\kappa_j}
 +\frac{\dot\psi^i}{\kappa_j}+\frac{\dot\psi^j}{\kappa_i}
 =\frac{\kappa_i^2\dot\psi^i-\kappa_j^2\dot\psi^j}
        {\kappa_i\kappa_j(\kappa_i-\kappa_j)}.
\end{equation*}
At repeated eigenvalues, we use continuous extension.

\subsection{Hypersurfaces in hyperbolic space}\label{subsec-geometry}

We use the standard conventions and identities for hypersurfaces in
space forms as in \cites{Gerhardt,ACGL} and
\cite{ACW}*{Section~2}. Fix $p\in\mathbb H^{n+1}$ and let $r$ be the distance from $p$.
In geodesic polar coordinates,
\begin{equation*}
 \overline g=dr^2+\sinh^2r\,\sigma,
\end{equation*}
where $\sigma$ is the standard round metric on $\mathbb S^n$.
We denote by $\overline\nabla$ and $D$ the Levi-Civita connections
of $\overline g$ and $\sigma$, respectively. The radial vector field
satisfies
\begin{equation*}
 \overline\nabla_U(\sinh r\,\partial_r)=\cosh r\,U.
\end{equation*}

Let $X$ be a smooth embedding of a closed hypersurface $M^n$
into $\mathbb H^{n+1}$ with outward unit normal $\nu$. Its induced metric, Levi-Civita
connection, and area element are denoted by $g$, $\nabla$, and
$d\mu$. In local coordinates, with
$\partial_i=X_*(\partial/\partial x^i)$, our conventions are
\begin{equation*}
 g_{ij}=\overline g(\partial_i,\partial_j),\qquad
 h_{ij}=\overline g(\overline\nabla_{\partial_i}\nu,\partial_j),
 \qquad \mathcal W=(h^i{}_j),\quad h^i{}_j=g^{ik}h_{kj}.
\end{equation*}
The Weingarten map $\mathcal W$ is $g$-self-adjoint, with
principal curvatures $\kappa_1\leq\cdots\leq\kappa_n$. We write
\begin{equation*}
 H=\operatorname{tr}\mathcal W=nE_1,\qquad
 (h^2)_{ij}=h_{ik}g^{k\ell}h_{\ell j}.
\end{equation*}
A geodesic sphere of radius $r$ has
$\mathcal W=\coth r\,\operatorname{Id}$. The Gauss and Codazzi
equations are
\begin{equation}\label{Gauss}
 R_{ijk\ell}
 =-(g_{ik}g_{j\ell}-g_{i\ell}g_{jk})
                         +h_{ik}h_{j\ell}-h_{i\ell}h_{jk},
\end{equation}
\begin{equation}\label{Codazzi}
 \nabla_i h_{jk}=\nabla_j h_{ik}.
\end{equation}
In particular, a principal two-plane has sectional curvature
$\kappa_i\kappa_j-1$, and $\nabla_i h_{jk}$ is totally symmetric.
Curvature-function derivatives on $M$ are evaluated in a local
orthonormal frame, or equivalently at the symmetric matrix
$g^{-1/2}hg^{-1/2}$. Thus \eqref{trace-bounds} applies with
$\delta_{ij}$ and $(A^2)_{ij}$ replaced by $g_{ij}$ and
$(h^2)_{ij}$, respectively.

Suppose that $M$ bounds a convex domain $\Omega$ and
$p\in\operatorname{int}\Omega$. The support function is
\begin{equation*}
 u=\overline g(\sinh r\,\partial_r,\nu).
\end{equation*}
The support-function identities \cite{HLW}*{Lemma~2.4} are
\begin{equation*}
\begin{aligned}
 \nabla_i u&=h^j{}_i\nabla_j\cosh r,\\
 \nabla_i\nabla_j u
 &=\nabla^k\cosh r\,\nabla_kh_{ij}
                         +\cosh r\,h_{ij}-u(h^2)_{ij}.
\end{aligned}
\end{equation*}
If $B_\rho(p)\subset\Omega$, then
\begin{equation}\label{support-lower}
 u\geq\sinh\rho.
\end{equation}

Writing $M$ as a radial graph $r=r(\theta)$ about $p$, with
$\theta\in\mathbb S^n$, set
\begin{equation*}
 r_i=D_i r,\qquad r_{ij}=D_iD_jr,\qquad
 v=\sqrt{1+\frac{|Dr|_\sigma^2}{\sinh^2r}}.
\end{equation*}
The radial graph identities \cite{HLW}*{Section~2.3} are
\begin{equation}\label{radial-graph}
\begin{aligned}
 g_{ij}&=\sinh^2r\,\sigma_{ij}+r_ir_j,\\
 h_{ij}&=\frac1v\bigl(-r_{ij}+\sinh r\cosh r\,\sigma_{ij}
                                      +2\coth r\,r_ir_j\bigr),\\
 u&=\frac{\sinh r}{v}.
\end{aligned}
\end{equation}

We also use the Klein model. In the hyperboloid realization,
the inverse projection is
\begin{equation}\label{Klein-parametrization}
 X=\frac{(1,Y)}{\sqrt{1-|Y|^2}},\qquad
 Y\in B_1(0)\subset\mathbb R^{n+1}.
\end{equation}
Geodesic segments correspond to Euclidean line segments, so
the Klein image of a convex domain is a Euclidean convex body.
Let $g_E,h_E$ be the Euclidean fundamental forms of an oriented
hypersurface in $B_1(0)$, with unit normal $\nu_E$, and put
$s=\langle Y,\nu_E\rangle$. The corresponding hyperbolic normal is
\begin{equation*}
 \nu_H=\frac{(s,\nu_E)}{\sqrt{1-s^2}}.
\end{equation*}
The pulled-back hyperbolic fundamental forms satisfy
\cite{AW}*{Section~5}
\begin{equation}\label{Klein-metric-comparison}
 (g_H)_{ij}
 =\frac1{1-|Y|^2}
   \left((g_E)_{ij}
      +\frac{\langle Y,Y_i\rangle\langle Y,Y_j\rangle}{1-|Y|^2}\right)
\end{equation}
and
\begin{equation}\label{Klein-second-form-comparison}
 h_H=\frac{h_E}{\sqrt{1-|Y|^2}\sqrt{1-s^2}}.
\end{equation}
These formulas give the following strict approximation.

\begin{lemma}\label{Klein-contraction}
Let $M\subset\mathbb H^{n+1}$ be a smooth closed strictly convex
hypersurface, $n\geq2$, and let $\Sigma$ be its Klein image.
For $0<c\leq1$, let $M_c$ have Klein image $c\Sigma$.
At corresponding points, the ordered principal curvatures satisfy
\begin{equation}\label{Klein-contraction-curvatures}
 \kappa_i(M_c)\geq c^{-1}\kappa_i(M),\qquad 1\leq i\leq n.
\end{equation}
In particular, if $\operatorname{sec}_M\geq0$, then
\begin{equation}\label{Klein-contraction-section}
 \operatorname{sec}_{M_c}\geq c^{-2}-1>0\quad(0<c<1),
 \qquad M_c\longrightarrow M\quad\text{smoothly as }c\uparrow1.
\end{equation}
\end{lemma}

\begin{proof}
Use the notation of \eqref{Klein-metric-comparison} and
\eqref{Klein-second-form-comparison}, with
$s=\langle Y,\nu_E\rangle$, and pull all forms back to $\Sigma$.
Under $Y_c=cY$,
\begin{equation*}
 g_{E,c}=c^2g_E,\qquad h_{E,c}=ch_E,\qquad
 \nu_{E,c}=\nu_E,\qquad s_c=cs.
\end{equation*}
Thus
\begin{equation}\label{Klein-contraction-forms}
\begin{aligned}
 g_c&\leq\frac{c^2(1-|Y|^2)}{1-c^2|Y|^2}\,g_1,\\
 h_c&=c\sqrt{\frac{(1-|Y|^2)(1-s^2)}
                   {(1-c^2|Y|^2)(1-c^2s^2)}}\,h_1.
\end{aligned}
\end{equation}
The metric inequality follows from
$c^2(1-|Y|^2)\leq1-c^2|Y|^2$. Since $h_1>0$, for $V\ne0$,
\begin{equation*}
 \frac{h_c(V,V)}{g_c(V,V)}
 \geq\frac1c\sqrt{\frac{(1-c^2|Y|^2)(1-s^2)}
                         {(1-|Y|^2)(1-c^2s^2)}}
          \frac{h_1(V,V)}{g_1(V,V)}
 \geq\frac1c\frac{h_1(V,V)}{g_1(V,V)},
\end{equation*}
where the last inequality uses
\begin{equation*}
 (1-c^2|Y|^2)(1-s^2)-(1-|Y|^2)(1-c^2s^2)
 =(1-c^2)(|Y|^2-s^2)\geq0.
\end{equation*}
The min--max principle proves \eqref{Klein-contraction-curvatures}.
Equation~\eqref{Gauss} and smoothness of the inverse Klein map give
\eqref{Klein-contraction-section}.
\end{proof}

\subsection{Hyperbolic quermassintegrals}\label{subsec-quermassintegrals}

For a smooth bounded convex domain $\Omega\subset\mathbb H^{n+1}$,
we use the normalization of \cite{ACW}*{(1.10)--(1.11)}:
\begin{equation}\label{W-definition}
\begin{aligned}
 W_0(\Omega)&=|\Omega|,\qquad
 W_1(\Omega)=\frac{|\partial\Omega|}{n},\\
 (n-j)W_{j+1}(\Omega)
 &=\int_{\partial\Omega}E_j\,d\mu-jW_{j-1}(\Omega),
 \qquad 1\leq j\leq n-1.
\end{aligned}
\end{equation}
The integral-geometric definition implies positivity, inclusion
monotonicity, and continuity under Hausdorff convergence of convex
bodies. We refer to \cites{Santalo,Solanes2006} and
\cite{ACW}*{Section~1.2}.
Under a normal variation $\partial_tX=\varphi\nu$, the first
variation formula \cite{ACW}*{Lemma~2.3} is
\begin{equation}\label{W-variation-intro}
 \frac{d}{dt}W_j(\Omega_t)=\int_{M_t}E_j\varphi\,d\mu_t,
 \qquad 0\leq j\leq n.
\end{equation}
For $f_j(r)=W_j(B_r)$, this gives
\begin{equation}\label{ball-functions}
 f_j'(r)=\omega_n\sinh^{n-j}r\cosh^j r>0,\qquad
 f_j(0)=0,\qquad \lim_{r\to\infty}f_j(r)=\infty,
\end{equation}
where $\omega_n=|\mathbb S^n|$ and $0\leq j\leq n$.
Thus $f_j^{-1}$ is well defined on $(0,\infty)$.

Along \eqref{flow}--\eqref{normalization}, we have
\begin{equation*}
 E_k=E_\ell F^{k-\ell},\qquad
 \int_{M_t}E_\ell(F-\phi)\,d\mu_t=0.
\end{equation*}
Substituting in \eqref{W-variation-intro} yields
\begin{equation}\label{W-monotonicity}
\begin{aligned}
 \frac{d}{dt}W_\ell(\Omega_t)&=0,\\
 \frac{d}{dt}W_k(\Omega_t)
 &=\int_{M_t}E_\ell F^{k-\ell}(\phi-F)\,d\mu_t\\
 &=-\int_{M_t}E_\ell
       \bigl(F^{k-\ell}-\phi^{k-\ell}\bigr)(F-\phi)\,d\mu_t\leq0.
\end{aligned}
\end{equation}
For a strictly convex hypersurface, the last integral vanishes
if and only if $F$ is spatially constant. The same identities
hold for $\ell=0$, with $E_0=1$ and $F=E_k^{1/k}$.

\subsection{Evolution equations}\label{subsec-evolution}
We write the flow \eqref{flow} as
\begin{equation}\label{speed}
 \partial_tX=-G\nu,\qquad G=F-\phi(t),
 \qquad F=\left(\frac{E_k}{E_\ell}\right)^{1/(k-\ell)}.
\end{equation}
The global term $\phi(t)$ is fixed in spatial and pointwise curvature
differentiations.
Thus
\begin{equation}\label{composite-derivatives}
 \dot G^{ij}=\dot F^{ij},\qquad
 \ddot G^{ij,k\ell}=\ddot F^{ij,k\ell}.
\end{equation}
Set
\begin{equation}\label{linearized-operator}
 \mathcal L=\dot G^{ij}\nabla_i\nabla_j
           =\dot F^{ij}\nabla_i\nabla_j.
\end{equation}
The following evolution equations follow from
\cite{ACW}*{(2.7)--(2.9), (2.12)--(2.13), and (4.4)},
using $\dot F^{ij}h_{ij}=F$.
They hold for any smooth function $\phi(t)$.

\begin{lemma}\label{evolution-lemma}
Along \eqref{speed}, the metric and area element satisfy
\begin{equation}\label{metric-evolution}
 \partial_tg_{ij}=-2Gh_{ij},\qquad
 \partial_tg^{ij}=2Gh^{ij},\qquad
 \partial_t(d\mu_t)=-GH\,d\mu_t.
\end{equation}
In local coordinates, the second fundamental form satisfies
\begin{equation}\label{h-uncommuted}
 \partial_th_{ij}=\nabla_i\nabla_jG
                              -G\bigl((h^2)_{ij}+g_{ij}\bigr)
\end{equation}
and
\begin{equation}\label{h-evolution}
\begin{aligned}
 \partial_th_{ij}={}&\dot F^{k\ell}\nabla_k\nabla_\ell h_{ij}
       +\ddot F^{k\ell,pq}\nabla_i h_{k\ell}\nabla_j h_{pq}\\
 &+\dot F^{k\ell}\bigl((h^2)_{k\ell}+g_{k\ell}\bigr)h_{ij}
       -(2F-\phi)\bigl((h^2)_{ij}+g_{ij}\bigr).
\end{aligned}
\end{equation}
The Weingarten map satisfies
\begin{equation}\label{Weingarten-evolution}
\begin{aligned}
 (\partial_t-\mathcal L)h^i{}_j={}&
       \ddot F^{k\ell,pq}\nabla^i h_{k\ell}\nabla_j h_{pq}\\
 &+\dot F^{k\ell}\bigl((h^2)_{k\ell}+g_{k\ell}\bigr)h^i{}_j
       -\phi h^i{}_kh^k{}_j+(\phi-2F)\delta^i_j.
\end{aligned}
\end{equation}
Moreover,
\begin{equation}\label{F-evolution}
 (\partial_t-\mathcal L)F
 =(F-\phi)\dot F^{ij}\bigl((h^2)_{ij}-g_{ij}\bigr),
\end{equation}
\begin{equation}\label{inverse-evolution}
 (\partial_t-\mathcal L)F^{-1}
 =-\frac2{F^3}\dot F^{ij}\nabla_iF\nabla_jF
       -\frac{F-\phi}{F^2}\dot F^{ij}\bigl((h^2)_{ij}-g_{ij}\bigr),
\end{equation}
and the support function about a fixed point satisfies
\begin{equation}\label{support-evolution}
 (\partial_t-\mathcal L)u
 =(\phi-2F)\cosh r+u\dot F^{ij}(h^2)_{ij}.
\end{equation}
\end{lemma}

For every $0\leq \ell<k\leq n$, the normal-graph equation for
\eqref{flow}--\eqref{normalization} is uniformly parabolic near
a smooth closed strictly convex hypersurface. Here $\ell=0$ means
$F=E_k^{1/k}$ and $\phi=|M_t|^{-1}\int_{M_t}F\,d\mu_t$.
The normalization $\phi$ is a smooth functional of the
graph and its first two derivatives, with positive
denominator. Its linearization contributes a bounded
finite-rank term on the base space $h^{2+\gamma}$,
with domain $h^{4+\gamma}$, where $0<\gamma<1$.
The local parabolic theory
\cite{Lunardi}*{Sections~8.1 and~8.3}
therefore gives a unique smooth short-time solution.
Nearby smooth initial hypersurfaces have solutions
on a common time interval, depending continuously
on the initial data in $C^2$.
After shortening the interval if necessary,
the solutions remain embedded and strictly convex. We denote by $[0,T_{\max})$ the maximal interval
on which the solution is smooth and strictly convex.

\section{Preservation of positive sectional curvature}
\label{sec-preservation}

We prove preservation of positive sectional curvature for any
positive smooth function $\phi(t)$.

\begin{theorem}\label{preservation}
	Let $n\geq3$ and $1\leq \ell<k\leq n-1$ satisfy $k+\ell\leq n$.
	Suppose that
	$M_t\subset\mathbb H^{n+1}$ is a smooth closed solution of
	\begin{equation*}
		\partial_tX=(\phi(t)-F)\nu,
		\qquad F=\left(\frac{E_k}{E_\ell}\right)^{1/(k-\ell)},
	\end{equation*}
	on $[0,T)$, with principal curvatures in $\Gamma_+$, where
	$\phi$ is a positive smooth function of time. If $M_0$ has
	positive sectional curvature, then $M_t$ has positive sectional
	curvature for every $t\in[0,T)$.
\end{theorem}

We use $G=F-\phi(t)$ and
$\mathcal L=\dot F^{ij}\nabla_i\nabla_j$
as in Section~\ref{subsec-evolution}.
Order the principal curvatures as
$0<\kappa_1\leq\cdots\leq\kappa_n$, and set
\begin{equation*}
S=\kappa_1\kappa_2-1.
\end{equation*}
By the Gauss equation, $S$ is the minimum
sectional curvature. We shall show that
$S$ remains positive.

The function $S$ need not be smooth at repeated principal
curvatures. We therefore use smooth lower supports, a method
also used for curvature eigenvalue estimates in
\cite{BryanIvakiScheuer2023}.  A smooth
function $\varphi$ is a lower support for $S$ at $(x_0,t_0)$ if
$\varphi\le S$ in a backward parabolic neighbourhood of this
point and $\varphi(x_0,t_0)=S(x_0,t_0)$. The function $\varphi$
is defined on an open neighbourhood of the contact point, but
the support inequality is required only for $t\le t_0$.
\begin{lemma}\label{lem:rotating}
	Let $\varphi$ be a smooth lower support for $S$ at $(x_0,t_0)$,
	where $t_0>0$. Choose a $g(t_0)$-orthonormal principal frame
	$\{e_1,\ldots,e_n\}$ at $x_0$, and set
	\begin{equation*}
		\mathcal J_i=\{p>2\mid\kappa_p>\kappa_i\},
		\qquad i=1,2.
	\end{equation*}
	For a smooth symmetric curvature function $\Theta$, define at
	this point
	\begin{equation}\label{eq:third-order-form}
		\begin{aligned}
			Q[\Theta]={}&
			\kappa_2\ddot\Theta^{ij,pq}
			\nabla_1h_{ij}\nabla_1h_{pq}
			+\kappa_1\ddot\Theta^{ij,pq}
			\nabla_2h_{ij}\nabla_2h_{pq}\\
			&-2\sum_{m=1}^n\dot\Theta^m
			\left(\nabla_mh_{11}\nabla_mh_{22}
			-(\nabla_mh_{12})^2\right)\\
			&+2\sum_{m=1}^n\dot\Theta^m
			\sum_{i=1}^2\sum_{p\in\mathcal J_i}
			\frac{\kappa_1\kappa_2}
			{\kappa_i(\kappa_p-\kappa_i)}
			(\nabla_mh_{ip})^2.
		\end{aligned}
	\end{equation}
	Then, for $m=1,\ldots,n$,
	\begin{equation}\label{eq:support-gradient-multiplicity}
		\kappa_2\nabla_mh_{11}+\kappa_1\nabla_mh_{22}
		=\nabla_m\varphi,
	\end{equation}
	and
	\begin{equation}\label{eq:support-zero-gap}
		\nabla_mh_{ip}=0
		\quad\text{if }i\in\{1,2\},\ p>2,\ \kappa_p=\kappa_i.
	\end{equation}
	Moreover, at the contact point,
	\begin{equation}\label{eq:rotating-inequality}
		(\partial_t-\mathcal L)\varphi\ge Q[G]+\mathcal R,
	\end{equation}
	where
	\begin{equation}\label{sectional-reaction-original}
		\begin{aligned}
			\mathcal R={}&
			2\dot F^{ij}\bigl((h^2)_{ij}+g_{ij}\bigr)
			\kappa_1\kappa_2
			-\phi\kappa_1\kappa_2(\kappa_1+\kappa_2)\\
			&+(\phi-2F)(\kappa_1+\kappa_2).
		\end{aligned}
	\end{equation}
\end{lemma}

\begin{proof}
	All tensor components and derivatives below are evaluated at
	$(x_0,t_0)$ in the chosen principal frame. For any orthonormal
	pair $w_1,w_2$, the Cauchy--Binet formula gives
	\begin{equation*}
		\det\bigl(h(w_i,w_j)\bigr)_{i,j=1}^2
		=
		\sum_{a<b}\kappa_a\kappa_b
		(w_1^aw_2^b-w_1^bw_2^a)^2
		\ge\kappa_1\kappa_2,
	\end{equation*}
	since the squared minors sum to one. Equality holds for
	$(w_1,w_2)=(e_1,e_2)$. Thus any smooth local orthonormal
	extension of this pair defines a function
	\begin{equation*}
		f=\det\bigl(h(w_i,w_j)\bigr)_{i,j=1}^2-1
	\end{equation*}
	satisfying $f\ge S\ge\varphi$, with equality at the contact
	point. Taking the first spatial derivative of $f-\varphi$
	there proves \eqref{eq:support-gradient-multiplicity}.

	To prove \eqref{eq:support-zero-gap}, fix $i\in\{1,2\}$ and
	$p>2$ with $\kappa_p=\kappa_i$, and let $j$ be the other
	index in $\{1,2\}$. Set
	\begin{equation*}
		v_\pm=\frac{e_i\pm e_p}{\sqrt2}.
	\end{equation*}
	Both orthonormal pairs $(v_+,e_j)$ and $(v_-,e_j)$ realize
	the minimum $\kappa_1\kappa_2$. Extend each pair smoothly and
	orthonormally, with vanishing spatial covariant derivatives
	at the contact point, and let
	\begin{equation*}
		f_\pm
		=
		h(v_\pm,v_\pm)h(e_j,e_j)-h(v_\pm,e_j)^2-1.
	\end{equation*}
	Then $f_\pm\ge S\ge\varphi$, with equality at $(x_0,t_0)$.
	Consequently,
	$\nabla_mf_+=\nabla_mf_-=\nabla_m\varphi$, and subtraction gives
	\begin{equation*}
		0=\nabla_mf_+-\nabla_mf_-
		=2\kappa_j\nabla_mh_{ip}.
	\end{equation*}
	Since $\kappa_j>0$, this proves \eqref{eq:support-zero-gap}.

	We now choose the extension used in the evolution calculation.
	Extend the principal frame to a smooth local
	$g(t)$-orthonormal frame $v_1,\ldots,v_n$ on the reference
	manifold, with $v_i=e_i$ at the contact point, and prescribe
	\begin{equation}\label{eq:frame-jet}
		\nabla_m v_i
		=
		-\sum_{p\in\mathcal J_i}
		\frac{\nabla_mh_{ip}}{\kappa_p-\kappa_i}\,e_p,
		\qquad i=1,2.
	\end{equation}
	Complete these spatial derivatives skew-symmetrically, with
	no rotation inside $\operatorname{span}\{e_1,e_2\}$ or its
	orthogonal complement. Only positive spectral gaps occur in
	\eqref{eq:frame-jet}. The derivatives corresponding to zero
	gaps vanish by \eqref{eq:support-zero-gap}. Empty sums are
	understood as zero. In the time direction, choose
	\begin{equation*}
		\partial_t v_i=G\mathcal W(v_i)
	\end{equation*}
	at the contact point, consistently with
	$\partial_tg=-2Gh$.

	For this frame, set
	\begin{equation*}
		f=\det\bigl(h(v_i,v_j)\bigr)_{i,j=1}^2-1.
	\end{equation*}
	At the contact point,
	\begin{equation*}
		\nabla_m\bigl(h(v_i,v_j)\bigr)=\nabla_mh_{ij},
		\qquad i,j\le2.
	\end{equation*}
	Extend the direction $e_m$ geodesically. Differentiating
	orthonormality gives
	\begin{equation*}
		\langle\nabla_m\nabla_m v_i,v_i\rangle
		=-|\nabla_m v_i|^2.
	\end{equation*}
	Together with \eqref{eq:frame-jet}, this yields
	\begin{equation*}
		\nabla_m\nabla_m\bigl(h(v_i,v_i)\bigr)
		=
		\nabla_m\nabla_mh_{ii}
		-2\sum_{p\in\mathcal J_i}
		\frac{(\nabla_mh_{ip})^2}{\kappa_p-\kappa_i},
		\qquad i=1,2.
	\end{equation*}
	The second derivative formula for the two-dimensional
	determinant therefore gives
	\begin{equation}\label{eq:rotating-spatial}
		\begin{aligned}
			\nabla_m\nabla_mf={}&
			\kappa_2\nabla_m\nabla_mh_{11}
			+\kappa_1\nabla_m\nabla_mh_{22}-2\sum_{i=1}^2\sum_{p\in\mathcal J_i}
			\frac{\kappa_1\kappa_2}
			{\kappa_i(\kappa_p-\kappa_i)}
			(\nabla_mh_{ip})^2\\
			&+2\left(\nabla_mh_{11}\nabla_mh_{22}
			-(\nabla_mh_{12})^2\right).
		\end{aligned}
	\end{equation}

	For the time derivative, use fixed local coordinates agreeing
	with the principal frame at $(x_0,t_0)$. The choice of
	$\partial_t v_i$ and the metric evolution imply
	\begin{equation*}
		\begin{aligned}
			\partial_t f
			&=
			\kappa_2\bigl(\partial_t h_{11}+2G\kappa_1^2\bigr)
			+\kappa_1\bigl(\partial_t h_{22}+2G\kappa_2^2\bigr)\\
			&=
			\kappa_2\partial_t h^1{}_1
			+\kappa_1\partial_t h^2{}_2.
		\end{aligned}
	\end{equation*}
	Contracting \eqref{eq:rotating-spatial} with
	$\dot G^m=\dot F^m$ and substituting
	\eqref{Weingarten-evolution}, we obtain
	\begin{equation*}
		(\partial_t-\mathcal L)f=Q[G]+\mathcal R
	\end{equation*}
	at the contact point.

	Finally, $f-\varphi\ge0$ in a backward parabolic
	neighbourhood and vanishes at $(x_0,t_0)$. The time derivative
	from the left and the spatial second derivative test give
	\begin{equation}\label{eq:rotating-time-comparison}
		\partial_t\varphi\ge\partial_t f,
		\qquad
		\nabla^2\varphi\le\nabla^2f.
	\end{equation}
	Since $(\dot G^{ij})$ is positive definite,
	\begin{equation*}
		(\partial_t-\mathcal L)\varphi
		\ge(\partial_t-\mathcal L)f
		=Q[G]+\mathcal R.
	\end{equation*}
	This proves \eqref{eq:rotating-inequality}.
\end{proof}

We estimate the derivative terms in
\eqref{eq:rotating-inequality} for lower supports with zero
spatial gradient, as required in the preservation argument. The proof uses the paired concavity inequality in
Lemma~\ref{key-lemma}, which is proved in
Section~\ref{sec-algebra}.

\begin{lemma}\label{lem:third-order-nonnegative}
	Under the assumptions of Lemma~\ref{lem:rotating}, suppose that
	$\nabla\varphi(x_0,t_0)=0$. Then
	\begin{equation}\label{gradient-nonnegative}
		Q[G]=Q[F]\ge\frac2F\left(
		\kappa_2(\nabla_1F)^2+\kappa_1(\nabla_2F)^2\right)\ge0.
	\end{equation}
	Consequently,
	\begin{equation}\label{eq:support-reaction-lower}
		(\partial_t-\mathcal L)\varphi\ge\mathcal R
	\end{equation}
	at the contact point.
\end{lemma}

\begin{proof}
	We work at the contact point throughout. Equation~\eqref{Codazzi}
	makes $\nabla_mh_{ij}$ symmetric in all three indices, and
	\eqref{eq:support-gradient-multiplicity} gives
	\begin{equation}\label{critical-gradient}
		\kappa_2\nabla_mh_{11}+\kappa_1\nabla_mh_{22}=0,
		\qquad m=1,\ldots,n.
	\end{equation}
	Set $\psi=-F^{-1}$ as in \eqref{psi-definition}. Since
	$G=F-\phi(t)$ and $\phi(t)$ is spatially constant, the chain
	rule gives
	\begin{equation}\label{gradient-transfer}
		Q[G]=Q[F]=F^2Q[\psi]
		+\frac2F\left(\kappa_2(\nabla_1F)^2
		+\kappa_1(\nabla_2F)^2\right).
	\end{equation}
	It suffices to prove that $Q[\psi]\ge0$.

    First suppose that the principal curvatures are distinct.
Applying \eqref{matrix-Hessian} with
$\xi_{ij}=\nabla_mh_{ij}$ for $m=1,2$, we expand
\eqref{eq:third-order-form} as
\begin{equation}\label{eq:third-order-expanded}
\begin{aligned}
 Q[\psi]={}&
 \kappa_2\sum_{i,j=1}^n
       \ddot\psi^{ij}\nabla_1h_{ii}\nabla_1h_{jj}
 +\kappa_1\sum_{i,j=1}^n
       \ddot\psi^{ij}\nabla_2h_{ii}\nabla_2h_{jj}\\
 &+2\kappa_2\sum_{i<j}
       \frac{\dot\psi^i-\dot\psi^j}{\kappa_i-\kappa_j}
       (\nabla_1h_{ij})^2
 +2\kappa_1\sum_{i<j}
       \frac{\dot\psi^i-\dot\psi^j}{\kappa_i-\kappa_j}
       (\nabla_2h_{ij})^2\\
 &-2\sum_{m=1}^n\dot\psi^m
       \nabla_mh_{11}\nabla_mh_{22}
 +2\sum_{m=1}^n\dot\psi^m(\nabla_mh_{12})^2\\
 &+2\kappa_2\sum_{m=1}^n\sum_{p=3}^n
       \frac{\dot\psi^m}{\kappa_p-\kappa_1}
       (\nabla_mh_{1p})^2\\
 &+2\kappa_1\sum_{m=1}^n\sum_{p=3}^n
       \frac{\dot\psi^m}{\kappa_p-\kappa_2}
       (\nabla_mh_{2p})^2.
\end{aligned}
\end{equation}
Here $\ddot\psi^{ij}$ denotes differentiation in the
eigenvalue variables. We group the terms according to the
independent components of $\nabla h$, placing all Codazzi
permutations of each component in the same group, and write
\begin{equation}\label{eq:third-order-splitting}
 Q[\psi]=\mathcal I_1+\mathcal I_2+\mathcal I_3.
\end{equation}

The first group consists of all terms involving
$\nabla_1h_{ii}$. It contains the first diagonal Hessian
sum in \eqref{eq:third-order-expanded}. By \eqref{Codazzi},
\begin{equation*}
 \nabla_2h_{12}=\nabla_1h_{22},
 \qquad
 \nabla_ph_{1p}=\nabla_1h_{pp},
 \quad p>2.
\end{equation*}
Thus the $(1,2)$ entry in the second off-diagonal Hessian
sum and the square with $m=2$ contribute
\begin{equation*}
 2\left(
 \kappa_1\frac{\dot\psi^1-\dot\psi^2}{\kappa_1-\kappa_2}
 +\dot\psi^2
 \right)(\nabla_1h_{22})^2
 =
 2\frac{\kappa_1\dot\psi^1-\kappa_2\dot\psi^2}
          {\kappa_1-\kappa_2}
       (\nabla_1h_{22})^2.
\end{equation*}
The product with $m=1$ and the rotations with $m=p>2$
give the remaining terms. Hence
\begin{equation}\label{eq:I1-expansion}
\begin{aligned}
 \mathcal I_1={}&
 \kappa_2\sum_{i,j=1}^n
       \ddot\psi^{ij}\nabla_1h_{ii}\nabla_1h_{jj}+2\frac{\kappa_1\dot\psi^1-\kappa_2\dot\psi^2}
             {\kappa_1-\kappa_2}
       (\nabla_1h_{22})^2\\
       &-2\dot\psi^1\nabla_1h_{11}\nabla_1h_{22}+2\kappa_2\sum_{p=3}^n
       \frac{\dot\psi^p}{\kappa_p-\kappa_1}
       (\nabla_1h_{pp})^2.
\end{aligned}
\end{equation}

The second group consists of all terms involving
$\nabla_2h_{ii}$. Besides the second diagonal Hessian sum,
it contains the $(1,2)$ entry in the first off-diagonal
Hessian sum, the square with $m=1$, the product with $m=2$,
and the corresponding rotations with $m=p>2$. Using
\begin{equation*}
 \nabla_1h_{12}=\nabla_2h_{11},
 \qquad
 \nabla_ph_{2p}=\nabla_2h_{pp},
 \quad p>2,
\end{equation*}
and
\begin{equation*}
 \kappa_2\frac{\dot\psi^1-\dot\psi^2}{\kappa_1-\kappa_2}
 +\dot\psi^1
 =
 \frac{\kappa_1\dot\psi^1-\kappa_2\dot\psi^2}
      {\kappa_1-\kappa_2},
\end{equation*}
we obtain
\begin{equation}\label{eq:I2-expansion}
\begin{aligned}
 \mathcal I_2={}&
 \kappa_1\sum_{i,j=1}^n
       \ddot\psi^{ij}\nabla_2h_{ii}\nabla_2h_{jj}+2\frac{\kappa_1\dot\psi^1-\kappa_2\dot\psi^2}
             {\kappa_1-\kappa_2}
       (\nabla_2h_{11})^2\\
        &-2\dot\psi^2\nabla_2h_{11}\nabla_2h_{22}+2\kappa_1\sum_{p=3}^n
       \frac{\dot\psi^p}{\kappa_p-\kappa_2}
       (\nabla_2h_{pp})^2.
\end{aligned}
\end{equation}

The remaining components are
\begin{equation*}
 \nabla_ph_{11},\quad
 \nabla_ph_{22},\quad
 \nabla_ph_{12}
 \quad(p>2),
 \qquad
 \nabla_ph_{1r},\quad
 \nabla_ph_{2r}
 \quad(3\leq p<r\leq n).
\end{equation*}
They form $\mathcal I_3$. For $i=1,2$ and $p>2$, \eqref{Codazzi}
gives $\nabla_ih_{ip}=\nabla_ph_{ii}$. The corresponding
off-diagonal Hessian term combines with the rotation having
$m=i$, using
\begin{equation*}
 \frac{\dot\psi^i-\dot\psi^p}{\kappa_i-\kappa_p}
 +\frac{\dot\psi^i}{\kappa_p-\kappa_i}
 =
 \frac{\dot\psi^p}{\kappa_p-\kappa_i}.
\end{equation*}
Together with the product having $m=p$, these give the
first sum in \eqref{eq:I3-expansion} below.

For the component
$\nabla_ph_{12}=\nabla_1h_{2p}=\nabla_2h_{1p}$, the two
off-diagonal Hessian terms, the square with $m=p$, and the
rotations with $(m,i)=(1,2)$ and $(2,1)$ give a coefficient
equal to twice
\begin{equation*}
\begin{aligned}
 &\kappa_2
       \frac{\dot\psi^2-\dot\psi^p}{\kappa_2-\kappa_p}
 +\kappa_1
       \frac{\dot\psi^1-\dot\psi^p}{\kappa_1-\kappa_p}
 +\dot\psi^p
 +\frac{\kappa_1\dot\psi^1}{\kappa_p-\kappa_2}
 +\frac{\kappa_2\dot\psi^2}{\kappa_p-\kappa_1}\\
 &\qquad=
 \frac{
   (\kappa_2-\kappa_1)
      (\kappa_1\dot\psi^1-\kappa_2\dot\psi^2)
   +(\kappa_p^2-\kappa_1\kappa_2)\dot\psi^p
      }
      {(\kappa_p-\kappa_1)(\kappa_p-\kappa_2)}.
\end{aligned}
\end{equation*}
Finally, for $i=1,2$ and $3\leq p<r\leq n$, use
\begin{equation*}
 \nabla_ih_{pr}=\nabla_ph_{ir}=\nabla_rh_{ip}.
\end{equation*}
The $(p,r)$ entry in the Hessian sum in direction $i$
and the two rotations with derivative directions $p$ and
$r$ give the last sum below. Altogether,
\begin{equation}\label{eq:I3-expansion}
\begin{aligned}
 \mathcal I_3={}&
 2\sum_{p=3}^n\dot\psi^p
 \left[
   \frac{\kappa_2}{\kappa_p-\kappa_1}(\nabla_ph_{11})^2
  +\frac{\kappa_1}{\kappa_p-\kappa_2}(\nabla_ph_{22})^2
  -\nabla_ph_{11}\nabla_ph_{22}
 \right]\\
 &+2\sum_{p=3}^n
 \frac{
   (\kappa_2-\kappa_1)
      (\kappa_1\dot\psi^1-\kappa_2\dot\psi^2)
   +(\kappa_p^2-\kappa_1\kappa_2)\dot\psi^p
      }
      {(\kappa_p-\kappa_1)(\kappa_p-\kappa_2)}
       (\nabla_ph_{12})^2\\
 &+2\sum_{i=1}^2\sum_{3\leq p<r\leq n}
       \frac{\kappa_1\kappa_2}{\kappa_i}
 \left[
   \frac{\dot\psi^p-\dot\psi^r}{\kappa_p-\kappa_r}
  +\frac{\dot\psi^p}{\kappa_r-\kappa_i}
  +\frac{\dot\psi^r}{\kappa_p-\kappa_i}
 \right](\nabla_ph_{ir})^2.
\end{aligned}
\end{equation}
No component with all three indices greater than two occurs
in \eqref{eq:third-order-expanded}. Thus these three groups
account for every term, and we estimate them in this order.

	\medskip
\noindent\textit{The group $\mathcal I_1$.}
	For a vector $v\in\mathbb R^n$ satisfying
	$\kappa_2v_1+\kappa_1v_2=0$, define
	\begin{equation}\label{eq:diagonal-quadratic-form}
		\begin{aligned}
			\mathcal B(v)={}&
			\sum_{i,j}\ddot\psi^{ij}v_iv_j
			+2\frac{\kappa_1^2\dot\psi^1-\kappa_2^2\dot\psi^2}
			{\kappa_1-\kappa_2}
			\left(\frac{v_1}{\kappa_1}\right)^2+2\sum_{p>2}\frac{\dot\psi^p}{\kappa_p-\kappa_1}v_p^2.
		\end{aligned}
	\end{equation}
 By \eqref{eq:I1-expansion} and \eqref{critical-gradient},
 \begin{equation}\label{eq:first-diagonal-group}
  \mathcal I_1=\kappa_2\mathcal B
       (\nabla_1h_{11},\ldots,\nabla_1h_{nn}).
 \end{equation}

	To show that $\mathcal B(v)\ge0$, set
	\begin{equation*}
		r=n-\ell,\qquad s=n-k,\qquad
		\lambda_i=\kappa_i^{-1},\qquad
		\xi_i=-\frac{v_i}{\kappa_i^2}.
	\end{equation*}
	Then $\lambda_1>\lambda_2\ge\cdots\ge\lambda_n>0$ and
	$\lambda_2\xi_1+\lambda_1\xi_2=0$. With $q_{r,s}=(\frac{\sigma_r}{\sigma_s}
    )^{\frac{1}{r-s}}$ as in
	Proposition~\ref{key-lemma}, \eqref{dual-quotient} gives
	\begin{equation*}
		F_*(\lambda)=\left(\frac{\binom ns}{\binom nr}\right)^{1/(r-s)}q_{r,s}(\lambda),
		\qquad \psi(\kappa)=-F_*(\lambda).
	\end{equation*}
	Consequently,
	\begin{equation*}
		\dot\psi^i=\frac{\partial_iF_*}{\kappa_i^2},
		\qquad
		\ddot\psi^{ij}
		=-\frac{\partial_{ij}F_*}{\kappa_i^2\kappa_j^2}
		-\frac{2\partial_iF_*}{\kappa_i^3}\delta_{ij}.
	\end{equation*}
	Substitution into \eqref{eq:diagonal-quadratic-form}, using
	the relation between $\xi_1$ and $\xi_2$, gives
	\begin{equation}\label{eq:diagonal-paired-form}
		\mathcal B(v)
		=-\partial_\xi^2F_*
		+2\sum_{i=2}^n\frac{\partial_iF_*}{\lambda_1-\lambda_i}\xi_i^2
		-\frac{2\partial_1F_*}{\lambda_1-\lambda_2}\xi_1^2
		\ge0.
	\end{equation}
	The last inequality follows from \eqref{paired-Jacobi} in
	Proposition~\ref{key-lemma}, since
        $1\leq s<r\leq n-1$ and $r+s=2n-k-\ell\geq n$. This proves
	$\mathcal I_1\ge0$.

	\medskip
\noindent\textit{The group $\mathcal I_2$.}
 By \eqref{critical-gradient},
 $\nabla_2h_{22}/\kappa_2=-\nabla_2h_{11}/\kappa_1$.
 Substituting this relation in \eqref{eq:I2-expansion} gives
	\begin{equation}\label{eq:second-diagonal-comparison}
		\begin{aligned}
			\frac{\mathcal I_2}{\kappa_1}
			={}&\mathcal B(\nabla_2h_{11},\ldots,\nabla_2h_{nn})+2\sum_{p>2}\dot\psi^p
			\left(\frac1{\kappa_p-\kappa_2}
			-\frac1{\kappa_p-\kappa_1}\right)
			(\nabla_2h_{pp})^2
			\ge0.
		\end{aligned}
	\end{equation}
	The first term is nonnegative by \eqref{eq:diagonal-paired-form},
	and every coefficient in the remaining sum is nonnegative.

	\medskip
\noindent\textit{The group $\mathcal I_3$.}
 The first sum in \eqref{eq:I3-expansion} is nonnegative, since
 \eqref{critical-gradient} gives
 $(\nabla_ph_{11})(\nabla_ph_{22})\leq0$ for $p>2$.
 For the second sum, \eqref{psi-properties} implies
 \begin{equation*}
  \dot\psi^1\geq\dot\psi^2\geq\dot\psi^p>0,
  \qquad \kappa_p^2\dot\psi^p\geq\kappa_2^2\dot\psi^2.
 \end{equation*}
 Hence its numerator satisfies
 \begin{equation}\label{eq:remaining-mixed-coefficient}
 \begin{aligned}
 & (\kappa_2-\kappa_1)
       (\kappa_1\dot\psi^1-\kappa_2\dot\psi^2)
       +(\kappa_p^2-\kappa_1\kappa_2)\dot\psi^p\\
 \geq&
  -(\kappa_2-\kappa_1)^2\dot\psi^2
       +(\kappa_2^2-\kappa_1\kappa_2)\dot\psi^2\\
  =&\kappa_1(\kappa_2-\kappa_1)\dot\psi^2\geq0.
 \end{aligned}
 \end{equation}
 For the last sum, the same identities give
 \begin{equation*}
 \begin{aligned}
 &\frac{\dot\psi^p-\dot\psi^r}{\kappa_p-\kappa_r}
   +\frac{\dot\psi^p}{\kappa_r-\kappa_i}
   +\frac{\dot\psi^r}{\kappa_p-\kappa_i}\geq
 \frac{\dot\psi^p-\dot\psi^r}{\kappa_p-\kappa_r}
   +\frac{\dot\psi^p}{\kappa_r}
   +\frac{\dot\psi^r}{\kappa_p}\geq0,
 \qquad i=1,2,\quad 3\leq p<r\leq n.
 \end{aligned}
 \end{equation*}
 Thus $\mathcal I_3\geq0$.

It remains to consider repeated principal curvatures.
Set $T_{ijk}=\nabla_i h_{jk}$ and, for $\varepsilon>0$, define
\begin{equation*}
 \kappa_i^\varepsilon=\kappa_i+i\varepsilon,\qquad
 T_{ijk}^\varepsilon
 =\left(
   \frac{\kappa_i^\varepsilon\kappa_j^\varepsilon
         \kappa_k^\varepsilon}
        {\kappa_i\kappa_j\kappa_k}
   \right)^{1/2}T_{ijk},
\end{equation*}
with no summation in this definition. By \eqref{Codazzi},
$T^\varepsilon$ is fully symmetric. Its zero components are
preserved, and $(\kappa^\varepsilon,T^\varepsilon)\to(\kappa,T)$.
Moreover, \eqref{critical-gradient} gives
\begin{equation*}
 \kappa_2^\varepsilon T_{m11}^\varepsilon
 +\kappa_1^\varepsilon T_{m22}^\varepsilon=\left(\frac{\kappa_m^\varepsilon}{\kappa_m}\right)^{1/2}
 \frac{\kappa_1^\varepsilon\kappa_2^\varepsilon}
      {\kappa_1\kappa_2}
 \bigl(\kappa_2T_{m11}+\kappa_1T_{m22}\bigr)=0.
\end{equation*}
Thus the preceding calculation applies to these data and gives
$Q^\varepsilon[\psi]\geq0$, where $Q^\varepsilon$ denotes
\eqref{eq:third-order-form} evaluated at $(\kappa^\varepsilon,T^\varepsilon)$.

For $i=1,2$, $p>2$, and every $m$, \eqref{eq:support-zero-gap}
implies
\begin{equation*}
 \frac{(T_{mip}^\varepsilon)^2}
      {\kappa_p^\varepsilon-\kappa_i^\varepsilon}
 \longrightarrow
 \begin{cases}
  \displaystyle\frac{T_{mip}^2}{\kappa_p-\kappa_i},
       &p\in\mathcal J_i,\\
  \quad\ 0\quad\ \ ,   &\kappa_p=\kappa_i.
 \end{cases}
\end{equation*}
In the second case the numerator vanishes for every $\varepsilon$.
The matrix derivatives of $\psi$ are smooth on the positive cone,
so all remaining terms converge. Hence
\begin{equation*}
 Q[\psi]=\lim_{\varepsilon\to0^+}Q^\varepsilon[\psi]\geq0.
\end{equation*}
Combining this with \eqref{gradient-transfer} proves
\eqref{gradient-nonnegative}. Equation~\eqref{eq:rotating-inequality}
then gives \eqref{eq:support-reaction-lower}.
\end{proof}

\begin{proof}[Proof of Theorem~\ref{preservation}]
	Euler's identity $F=\sum_i\dot F^i\kappa_i$ allows us to
	rewrite \eqref{sectional-reaction-original} as
	\begin{equation}\label{sectional-reaction}
		\begin{aligned}
			\mathcal R={}&
			2\sum_i\dot F^i(\kappa_i-\kappa_1)(\kappa_i-\kappa_2)+S\left[2\dot F^{ij}(h^2)_{ij}
			-\phi(\kappa_1+\kappa_2)\right].
		\end{aligned}
	\end{equation}
	The first sum is nonnegative. Fix $\tau<T$. Smoothness and
	compactness give a finite constant
	\begin{equation*}
		C_\tau
		=\max_{M\times[0,\tau]}\phi(\kappa_1+\kappa_2).
	\end{equation*}
	Since $\dot F^{ij}(h^2)_{ij}>0$, it follows that
	\begin{equation}\label{eq:reaction-linear-lower}
		\mathcal R\ge-C_\tau S
		\qquad\text{whenever }S\ge0.
	\end{equation}

Choose $0<\varepsilon_0<\min_M S(\cdot,0)$ and set
\begin{equation*}
b(t)=\varepsilon_0e^{-(C_\tau+1)t}.
\end{equation*}
Suppose that $S$ first touches $b$ at
$(x_0,t_0)$, where $0<t_0\leq\tau$.
Then $\varphi(x,t)=b(t)$ is a smooth lower
support for $S$ with $\nabla\varphi=0$.
Lemmas~\ref{lem:rotating} and
\ref{lem:third-order-nonnegative}, together with
\eqref{eq:reaction-linear-lower}, give
\begin{equation*}
-(C_\tau+1)b(t_0)
=(\partial_t-\mathcal L)\varphi(x_0,t_0)
\geq\mathcal R(x_0,t_0)
\geq-C_\tau b(t_0),
\end{equation*}
a contradiction.
Thus $S\geq b(t)>0$ on $M\times[0,\tau]$.
Since $\tau<T$ is arbitrary, the theorem follows.
\end{proof}

\begin{remark}\label{spherical-flow-remark}
For $1\leq\ell<k$ with $k+\ell\leq n$, let $F$ be
defined by \eqref{normalization}. Consider the flow
\begin{equation*}
 \partial_tX=\left(\frac{a(t)}F-F\right)\nu,
 \qquad
 a(t)=
 \frac{\displaystyle\int_{M_t}E_\ell F\,d\mu_t}
      {\displaystyle\int_{M_t}E_\ell/F\,d\mu_t}.
\end{equation*}
For $k=\ell+1$, this is the hyperbolic analogue of the
flow studied in \cite{Spherical}. The choice of $a(t)$
preserves $W_\ell$ and makes $W_k$ nonincreasing,
by \eqref{W-variation-intro}. Positive sectional curvature is also preserved along
smooth strictly convex solutions. Indeed, at a lower-support
contact with $\nabla\varphi=0$ and $S\geq0$, the proof of
Lemma~\ref{lem:third-order-nonnegative} gives
$Q[F-a/F]=Q[F]+aQ[-F^{-1}]\geq0$.
Writing
$\widehat{\mathcal L}
=(1+a/F^2)\dot F^{ij}\nabla_i\nabla_j$,
the reaction terms yield
\begin{equation*}
 (\partial_t-\widehat{\mathcal L})\varphi
 \geq-\frac{2a}{F}(\kappa_1+\kappa_2)S.
\end{equation*}
The coefficient on the right is bounded on each compact
smooth time interval, so the barrier argument in the proof
of Theorem~\ref{preservation} applies.
\end{remark}

\section{Long-time existence and convergence}\label{sec-proof}
Let $[0,T_{\max})$ be the maximal interval of the solution starting
from $M_0$ as in Theorem~\ref{main-flow}.
Theorem~\ref{preservation} gives positive sectional curvature on
this interval. We use \eqref{W-monotonicity} to obtain the geometric
estimates, long-time existence, and convergence.

\subsection{Shape estimate}\label{subsec-shape}
For a bounded convex domain $\Omega$, let $\rho_-(\Omega)$ and $\rho_+(\Omega)$ denote its inner and outer radii. We first obtain a fixed outer ball, using the same reflection principle as \cite{ACW}*{Section 4.1}. The preserved quermassintegral replaces the enclosed volume in the containment argument.

\begin{lemma}\label{outer-ball}
There are a fixed point $p\in\mathbb H^{n+1}$ and a constant
$R_0$, depending only on the initial domain and $\ell$, such that
\begin{equation*}
 \Omega_t\subset B_{R_0}(p)
\end{equation*}
for every $t\in[0,T_{\max})$.
\end{lemma}

\begin{proof}
We use the Alexandrov reflection argument in the proof of
\cite{ACW}*{Lemma~4.1}. Choose $R>0$ such that
\begin{equation*}
 \overline\Omega_0\subset B_R(p),
 \qquad
 W_\ell(\Omega_0)<W_\ell(B_R(p)).
\end{equation*}
Since the global term $\phi(t)$ is spatially constant and invariant
under hyperbolic reflections, the reflection argument of
\cite{ACW}*{Lemma~4.1} applies without change.

It follows first that $\Omega_t$ meets $B_R(p)$ for every
$t<T_{\max}$. Moreover, $M_t$ must meet $B_R(p)$. Otherwise
$B_R(p)\subset\Omega_t$, and the inclusion monotonicity and
preservation of $W_\ell$ would give
\begin{equation*}
 W_\ell(B_R(p))
 \leq W_\ell(\Omega_t)
 =W_\ell(\Omega_0),
\end{equation*}
a contradiction. The remaining reflection argument in
\cite{ACW}*{Lemma~4.1} gives a uniform diameter bound, and hence
\begin{equation*}
 \Omega_t\subset B_{R_0}(p)
\end{equation*}
for a constant $R_0$ depending only on $\Omega_0$ and $\ell$.
\end{proof}

The following noncollapsing estimate extends the corresponding
argument in \cite{ACW}*{Lemma~4.1}, where preservation of the enclosed volume
is used to obtain an inradius lower bound. Here we replace volume
preservation by a positive lower bound for a general
quermassintegral $W_j$, using in addition the nonnegative
sectional curvature condition.

\begin{lemma}\label{noncollapsing}
Let $n\geq2$, let $D,A>0$, and let $j\in\{0,\ldots,n-1\}$.
There exists a constant
\begin{equation*}
 r_0=r_0(n,D,A,j)>0
\end{equation*}
with the following property. Suppose that $\Omega\subset\mathbb H^{n+1}$
is a bounded convex domain with smooth closed boundary $M=\partial\Omega$.
Let $\kappa_1,\ldots,\kappa_n$ be the principal curvatures of $M$ with
respect to the outward unit normal. Assume that
\begin{equation*}
 \operatorname{diam}_{\mathbb H^{n+1}}\Omega\leq D,
 \qquad W_j(\Omega)\geq A,
\end{equation*}
and that $M$ has nonnegative intrinsic sectional curvature, namely,
\begin{equation*}
 \kappa_i\kappa_\ell\geq1
 \qquad\text{on }M\quad\text{for all }i\ne\ell.
\end{equation*}
Then the inner radius
\begin{equation*}
 \rho_-(\Omega)
 =\sup\{r>0\mid B_r(p)\subset\Omega
                \text{ for some }p\in\Omega\}
\end{equation*}
satisfies
\begin{equation*}
 \rho_-(\Omega)>r_0.
\end{equation*}
In particular, $\Omega$ contains a geodesic ball of radius $r_0$.
\end{lemma}

\begin{proof}
Fix $p\in\Omega$ and choose the Klein coordinates
\eqref{Klein-parametrization} centred at $p$.
The image $K\subset B_1(0)$ of $\Omega$ is a Euclidean convex body,
and the diameter bound gives
\begin{equation*}
 |Y|\leq\rho=\tanh D<1\quad\text{on }K.
\end{equation*}
Since $s^2\leq|Y|^2$, \eqref{Klein-metric-comparison} and
\eqref{Klein-second-form-comparison} imply
\begin{equation*}
 C^{-1}g_E\leq g_H\leq Cg_E,\qquad
 C^{-1}h_E\leq h_H\leq Ch_E,
\end{equation*}
where $C$ depends only on $n$ and $D$.
Let $0<\widehat\kappa_1\leq\cdots\leq\widehat\kappa_n$
be the Euclidean principal curvatures of $\partial K$.
The min--max principle and $\kappa_1\kappa_2\geq1$ give
\begin{equation}\label{Euclidean-section-bound}
 \widehat\kappa_1\widehat\kappa_2\geq c(D)>0.
\end{equation}

Let $W_j^E$ denote the Euclidean quermassintegrals. For $j\geq1$,
the recurrence \eqref{W-definition}, the positivity of the lower
order quermassintegrals, and
\eqref{Klein-metric-comparison}--\eqref{Klein-second-form-comparison}
give
\begin{equation}\label{W-comparison}
 W_j(\Omega)
 \leq
 \frac1{n+1-j}
 \int_{\partial\Omega}E_{j-1}\,d\mu_H
 \leq
 C(n,D)W_j^E(K).
\end{equation}
For $j=0$, the same estimate follows directly from the uniform
comparison of the hyperbolic and Euclidean volume elements.

By John's ellipsoid theorem \cite{GS}, there is an ellipsoid
$E$, centred at an interior point of $K$, such that
\begin{equation*}
 E\subset K\subset (n+1)E.
\end{equation*}
Let $0<a_1\leq\cdots\leq a_{n+1}$ be its semiaxes, and set $\varepsilon=a_1$, $b=a_n$. The longest semiaxis is uniformly bounded. Intersect $K$ with the three-dimensional affine subspace through the centre of $E$ spanned by the shortest and the two longest axes. The boundary of this section is a smooth strictly convex surface. Its second fundamental form is the restriction of $h_E$, divided by the length of the projection of $\nu_E$ onto the three-dimensional subspace. Hence \eqref{Euclidean-section-bound} implies that its
Gauss curvature is bounded below by a positive constant depending only on $D$.

In the plane spanned by the two longest axes, the projection of
the section contains a disk $D_b$ of radius $b$. The upper part
of its boundary can be written as a concave graph
\begin{equation*}
 z=f(x),
 \qquad x\in D_b,
\end{equation*}
with
\begin{equation*}
 0\leq f\leq (n+1)\varepsilon.
\end{equation*}
Since the Gauss curvature of the graph is bounded below by
$c(D)>0$,
\begin{equation*}
 \frac{\det(-D^2f)}{(1+|Df|^2)^2}\geq c(D).
\end{equation*}
By the arithmetic--geometric mean inequality,
\begin{equation*}
 -\Delta f
 \geq
 2\sqrt{\det(-D^2f)}
 \geq2\sqrt{c(D)}.
\end{equation*}
For $0<R<b$, consider
\begin{equation*}
 q_R(x)
 =
 \frac{\sqrt{c(D)}}2(R^2-|x|^2).
\end{equation*}
Then
\begin{equation*}
 -\Delta q_R=2\sqrt{c(D)},
 \qquad
 q_R=0\quad\text{on }\partial D_R.
\end{equation*}
Since $f\geq0$ on $\partial D_R$, the maximum principle gives
$f\geq q_R$ in $D_R$. In particular,
\begin{equation*}
 (n+1)\varepsilon
 \geq f(0)
 \geq\frac{\sqrt{c(D)}}2R^2.
\end{equation*}
Letting $R\uparrow b$, we obtain
\begin{equation}\label{semiaxis-control}
 b^2\leq C(n,D)\varepsilon.
\end{equation}

It follows from John's inclusion that $K$ is contained in the
$C\sqrt{\varepsilon}$-neighbourhood of a line segment $I$ of
uniformly bounded length. Let $\beta_m$ denote the volume of the
Euclidean unit ball in $\mathbb R^m$. The Steiner formula (see~\cite{Schneider}*{Section~4.2}) gives
\begin{equation*}
 W_j^E(I+B_s)
 =
 \beta_{n+1}s^{n+1-j}
 +
 |I|\beta_n\frac{n+1-j}{n+1}s^{n-j}.
\end{equation*}
Using Euclidean inclusion monotonicity,
\eqref{W-comparison}, and \eqref{semiaxis-control}, we obtain
\begin{equation*}
 A
 \leq W_j(\Omega)
 \leq
 C\left(
 \varepsilon^{(n+1-j)/2}
 +
 \varepsilon^{(n-j)/2}
 \right).
\end{equation*}

Since $0<\varepsilon\leq\rho<1$ and $j\leq n-1$, this gives
\begin{equation*}
 A\leq C\varepsilon^{(n-j)/2},
 \qquad
 \varepsilon\geq\left(\frac AC\right)^{2/(n-j)}>0.
\end{equation*}

If $z_E$ denotes the centre of the John ellipsoid, then $B_\varepsilon^E(z_E)\subset K$. On the fixed Klein ball $B_\rho(0)$ the hyperbolic and Euclidean metrics are uniformly equivalent. Therefore $K$ contains a hyperbolic ball of radius at least $c(D)\varepsilon$. This gives
\begin{equation*}
 \rho_-(\Omega)\geq r_0(n,D,A,j)>0
\end{equation*}
and proves the lemma.
\end{proof}

\begin{remark}
The sectional curvature assumption in
Lemma~\ref{noncollapsing} is essential when
$1\leq j\leq n-1$.
To see this, fix $0<a<1$ and consider the
following ellipsoids in the Klein model:
\begin{equation*}
 E_\varepsilon=
 \left\{
 x\in\mathbb R^{n+1}\ \middle|\
 \frac{x_1^2+\cdots+x_{n+1-j}^2}{a^2}
 +\frac{x_{n+1-j+1}^2+\cdots+x_{n+1}^2}{\varepsilon^2}
 \leq1
 \right\},
\end{equation*}
where $\varepsilon>0$ is small. Let $\Omega_\varepsilon$
denote the corresponding convex domains in $\mathbb H^{n+1}$.
Since the ellipsoids remain in a fixed compact subset of the
Klein ball, their hyperbolic diameters are uniformly bounded,
whereas $\rho_-(\Omega_\varepsilon)\to 0$. On the other hand, $E_\varepsilon$ converges in the Hausdorff topology to a $(n+1-j)$-dimensional disk. By the Crofton description and continuity of the quermassintegrals \cites{Santalo,Solanes2006},
\begin{equation*}
 W_j(\Omega_\varepsilon)\longrightarrow c_{a,j}>0.
\end{equation*}
Thus a uniform diameter bound and a positive lower bound for
$W_j$ do not prevent collapse under convexity alone. The case
$j=0$ is different, since a positive volume lower bound already
excludes such collapse, as in \cite{ACW}*{Lemma~4.1}.
\end{remark}

\begin{corollary}\label{shape-bounds}
Under the assumptions of Theorem~\ref{main-flow}, there exist
constants $r_0,R_0>0$, depending only on $n$, $k$, $\ell$ and $M_0$,
such that
\begin{equation}
 0<r_0\leq\rho_-(\Omega_t)\leq\rho_+(\Omega_t)\leq R_0,
 \qquad
 \Omega_t\subset B_{R_0}(p),
\end{equation}
for every $t\in[0,T_{\max})$, where $p$ is the fixed point
given by Lemma~\ref{outer-ball}.
\end{corollary}

\begin{proof}
The conclusion follows from Theorem~\ref{preservation},
Lemmas~\ref{outer-ball} and~\ref{noncollapsing}, and the
preservation of $W_\ell$. Since $1\leq \ell<k\leq n-1$, we have $\ell\leq n-2$,
so Lemma~\ref{noncollapsing} applies with $j=\ell$.
\end{proof}

\begin{lemma}\label{fixed-center}
There is $\tau_0>0$, depending only on the data in Theorem~\ref{main-flow}, with the following property. Given $t_0<T_{\max}$, choose $B_{r_0}(p_0)\subset\Omega_{t_0}$. Then
\begin{equation*}
B_{r_0/4}(p_0)\subset\Omega_t
\quad\text{for }t_0\leq t<\min\{T_{\max},t_0+\tau_0\}.
\end{equation*}
On this time interval, the support function relative to $p_0$ satisfies
\begin{equation}\label{fixed-support-bounds}
\sinh(r_0/4)\leq u\leq\sinh(2R_0),\qquad r\leq2R_0.
\end{equation}
\end{lemma}
\begin{proof}
Consider the shrinking geodesic ball with radius $\rho$ solving
\begin{equation*}
\rho'=-\coth\rho,\qquad \rho(t_0)=r_0/2.
\end{equation*}
It satisfies $\cosh\rho(t)=e^{-(t-t_0)}\cosh(r_0/2)$. Take
\begin{equation*}
\tau_0=\log\frac{\cosh(r_0/2)}{\cosh(r_0/4)}.
\end{equation*}
At a first contact with the enclosing hypersurface, its Weingarten matrix is at most $\coth\rho\,I$. Since $F$ is increasing and normalized,
\begin{equation*}
\phi-F>-F\geq-\coth\rho=\rho',
\end{equation*}
contradicting a first crossing. Thus the ball remains inside. The support lower bound follows from \eqref{support-lower}, and the upper bounds follow from the fixed outer ball and $p_0\in\Omega_{t_0}$.
\end{proof}

\subsection{Upper bound of \texorpdfstring{$F$}{F}}\label{subsec-F-bound}
We first record consequences of positive sectional curvature.
If $a=\kappa_1<1$, then $\kappa_i\geq a^{-1}$ for $i\geq2$.
For $\ell\leq j\leq k-1\leq n-2$, monotonicity of the adjacent
quotients gives
\begin{equation*}
 \frac{E_{j+1}}{E_j}
 \geq\frac{n-j-1+(j+1)a^2}{a(n-j+ja^2)}
 =1+\frac{(1-a)(ja^2-a+n-j-1)}{a(n-j+ja^2)}
 \geq1.
\end{equation*}
If $\kappa_1\geq1$, the same bound follows from monotonicity.
By \eqref{quotient-product} and \eqref{normalization},
\begin{equation}\label{F-phi-lower}
 F\geq1,\qquad \phi\geq1.
\end{equation}
Moreover, setting
\begin{equation*}
 \beta_{n,k,\ell}
 =\left(\frac{n-k}{n-\ell}\right)^{1/(k-\ell)}>0,
\end{equation*}
we have, whenever $a=\kappa_1<1$,
\begin{equation}\label{quotient-test-spectrum}
 F\geq F(a,a^{-1},\ldots,a^{-1})
 =\frac1a
   \left(\frac{n-k+ka^2}{n-\ell+\ell a^2}\right)^{1/(k-\ell)}
 \geq\frac{\beta_{n,k,\ell}}a.
\end{equation}
Since every factor in \eqref{quotient-product} is at least one,
\begin{equation*}
 \frac{E_k}{E_{k-1}}\geq1,
 \qquad
 \frac{E_{\ell+1}}{E_\ell}\leq F^{k-\ell}.
\end{equation*}
The identities \eqref{quotient-trace-identities} therefore imply
\begin{equation}\label{sectional-trace-bounds}
 \dot F^{ij}g_{ij}\leq\frac{k}{k-\ell}F,
 \qquad
 \dot F^{ij}(h^2)_{ij}
 \leq\frac{n-\ell}{k-\ell}F^{k-\ell+1}.
\end{equation}

\begin{proposition}\label{F-upper}
There exists a constant $C_F>0$, depending only on $n$, $k$, $\ell$
and $M_0$, such that, for every $t\in[0,T_{\max})$,
\begin{equation}\label{F-phi-bounds}
 1\leq F\leq C_F,
 \qquad
 1\leq\phi(t)\leq C_F.
\end{equation}
\end{proposition}

\begin{proof}
Fix $t_0<T_{\max}$ and choose an interior center as in
Lemma~\ref{fixed-center}. On the interval
$[t_0,\min\{T_{\max},t_0+\tau_0\})$, let $u$ be the support
function relative to this center and set
\begin{equation*}
 c=\frac12\sinh(r_0/4),
 \qquad
 Z=\frac{F}{u-c}.
\end{equation*}
By \eqref{fixed-support-bounds},
\begin{equation*}
 c\leq u-c\leq\sinh(2R_0),
 \qquad r\leq2R_0.
\end{equation*}
At a spatial maximum of $Z$, we have
\begin{equation*}
 \nabla F=\frac{F}{u-c}\nabla u.
\end{equation*}
Using \eqref{F-evolution} and \eqref{support-evolution},
the gradient terms in the quotient calculation cancel,
and we obtain
\begin{equation}\label{Tso-inequality}
\begin{aligned}
 (\partial_t-\mathcal L)Z
 ={}&
 -\frac{cF+\phi(u-c)}{(u-c)^2}\dot F^{ij}(h^2)_{ij}-\frac{F-\phi}{u-c}\dot F^{ij}g_{ij}
 +\frac{F(2F-\phi)}{(u-c)^2}\cosh r.
\end{aligned}
\end{equation}
The terms containing $\phi$ combine into
\begin{equation*}
 \frac{\phi}{(u-c)^2}
 \left[
 (u-c)\bigl(\dot F^{ij}g_{ij}
                   -\dot F^{ij}(h^2)_{ij}\bigr)
 -F\cosh r
 \right].
\end{equation*}
By \eqref{trace-bounds} and \eqref{sectional-trace-bounds},
this expression is nonpositive whenever $F\geq k/(k-\ell)$. The remaining term
$-F\dot F^{ij}g_{ij}/(u-c)$ is also nonpositive. Hence,
at a spatial maximum where $Z\geq k/((k-\ell)c)$, we have
\begin{equation}\label{Tso-ODE}
\begin{aligned}
 (\partial_t-\mathcal L)Z
 &\leq
 -\frac{cF^3}{(u-c)^2}
 +2\cosh r\,\frac{F^2}{(u-c)^2}\\
 &\leq
 -c^2Z^3+2\cosh(2R_0)Z^2.
\end{aligned}
\end{equation}

Choose a constant $Z_*>0$, depending only on $n$, $k$, $\ell$,
$r_0$ and $R_0$, such that
\begin{equation*}
 Z_*\geq\frac{k}{(k-\ell)c},
 \qquad
 -c^2z^3+2\cosh(2R_0)z^2
 \leq-\frac{c^2}{2}z^3
 \quad\text{for }z\geq Z_*.
\end{equation*}
For $t>t_0$, the function
\begin{equation*}
 z(t)=Z_*+\frac1{c\sqrt{t-t_0}}
\end{equation*}
satisfies
\begin{equation*}
 z'(t)
 =-\frac{c^2}{2}(z-Z_*)^3
 >-\frac{c^2}{2}z^3,
 \qquad
 \lim_{t\downarrow t_0}z(t)=+\infty.
\end{equation*}
At a first contact of $Z$ with $z$ from below, $Z$ attains
a spatial maximum and \eqref{Tso-ODE} contradicts the
preceding strict inequality. The comparison principle
therefore gives
\begin{equation*}
 Z(x,t)\leq Z_*+\frac1{c\sqrt{t-t_0}},
 \qquad
 t_0<t<\min\{T_{\max},t_0+\tau_0\}.
\end{equation*}

For each $t\geq\tau_0/2$ with $t<T_{\max}$, apply this estimate
on the interval starting at $t_0=t-\tau_0/2$. It follows that
\begin{equation*}
 F(x,t)
 \leq\sinh(2R_0)
       \left(Z_*+\frac{\sqrt2}{c\sqrt{\tau_0}}\right).
\end{equation*}
On the initial interval, use the center chosen at $t_0=0$.
Comparison with a constant in \eqref{Tso-ODE} yields
\begin{equation*}
 Z(x,t)\leq
 \max\left\{Z_*,\max_{M_0}Z(\cdot,0)\right\},
 \qquad
 0\leq t<\min\{T_{\max},\tau_0\}.
\end{equation*}
Since $u-c\geq c$ initially and $u-c\leq\sinh(2R_0)$
throughout this interval, these estimates give
$F\leq C_F$ on $[0,T_{\max})$, with $C_F$ depending only
on $n$, $k$, $\ell$ and $M_0$.

Finally, \eqref{normalization} gives
\begin{equation*}
 \phi(t)
 =
 \frac{\displaystyle\int_{M_t}E_\ell F\,d\mu_t}
      {\displaystyle\int_{M_t}E_\ell\,d\mu_t}.
\end{equation*}
Thus $\phi(t)\leq C_F$. Together with
\eqref{F-phi-lower}, this proves \eqref{F-phi-bounds}.
\end{proof}

\subsection{Curvature estimates and long-time existence}\label{subsec-curvature}
\begin{proposition}\label{two-sided-curvature}
There exist positive constants $c_\kappa$ and $C_\kappa$,
depending only on $n$, $k$, $\ell$ and $M_0$, such that
\begin{equation*}
 c_\kappa\leq\kappa_i(x,t)\leq C_\kappa,
 \qquad i=1,\ldots,n,
 \qquad (x,t)\in M\times[0,T_{\max}).
\end{equation*}
\end{proposition}

\begin{proof}
We first prove the lower bound. At a point where
$a=\kappa_1<1$, positive sectional curvature implies
$\kappa_i\geq a^{-1}$ for every $i\geq2$. By
\eqref{quotient-test-spectrum} and Proposition~\ref{F-upper},
\begin{equation*}
 C_F\geq F
 \geq
 \frac1a\left(\frac{n-k+ka^2}{n-\ell+\ell a^2}\right)^{1/(k-\ell)}
 \geq\frac{\beta_{n,k,\ell}}a.
\end{equation*}
Since $k\leq n-1$, we obtain
\begin{equation}\label{kappa-lower}
 \kappa_1\geq
 c_\kappa=
 \min\left\{1,\frac{\beta_{n,k,\ell}}{C_F}\right\}>0.
\end{equation}
The bound also holds where $\kappa_1\geq1$.

For the upper bound, choose
$C_\kappa>\max_{M_0}\kappa_n$ sufficiently large that
\begin{equation*}
 \frac{(n-\ell)C_F^{k-\ell+1}+kC_F}{k-\ell}C_\kappa
 -C_\kappa^2+C_F<0.
\end{equation*}
Suppose that $\kappa_n$ first reaches $C_\kappa$ at
$(x_0,t_0)$, with $t_0<T_{\max}$. Choose an orthonormal
principal frame at this point, and extend $e_n$ to a smooth
local $g(t)$-unit vector field $v$ satisfying
\begin{equation*}
 v=e_n,\qquad
 \nabla_i v=0,\qquad
 \partial_t v=G\mathcal W(v)
 \quad\text{at }(x_0,t_0).
\end{equation*}
These choices are compatible with
$\partial_tg=-2Gh$. Define
\begin{equation*}
 f=h(v,v).
\end{equation*}
Then $f\leq\kappa_n$, with equality at $(x_0,t_0)$.
Thus $f\leq C_\kappa$ in a backward parabolic neighbourhood
of this point and $f(x_0,t_0)=C_\kappa$. The first-contact
test gives
\begin{equation*}
 \partial_t f\geq0,\qquad
 \nabla^2f\leq0,\qquad
 (\partial_t-\mathcal L)f\geq0
 \quad\text{at }(x_0,t_0).
\end{equation*}

All subsequent calculations are made at the contact point.
Since $v$ is unit, $\nabla v=0$, and
$\mathcal W(v)=\kappa_n v$, differentiation gives
\begin{equation*}
 \mathcal Lf
 =\dot F^{ij}\nabla_i\nabla_jh_{nn},
 \qquad
 \partial_t f
 =\partial_t h_{nn}+2G\kappa_n^2
 =\partial_t h^n{}_n.
\end{equation*}
Here the time derivatives of the tensor components are taken
in fixed local coordinates agreeing with the principal frame
at $(x_0,t_0)$. Applying \eqref{Weingarten-evolution}, we obtain
\begin{equation*}
\begin{aligned}
 (\partial_t-\mathcal L)f
 ={}&
 \ddot F^{ij,pq}
       \nabla_nh_{ij}\nabla_nh_{pq}+\dot F^{ij}\bigl((h^2)_{ij}+g_{ij}\bigr)f
 -\phi f^2+\phi-2F.
\end{aligned}
\end{equation*}
The function $F$ is concave, so
\begin{equation*}
 \ddot F^{ij,pq}
       \nabla_nh_{ij}\nabla_nh_{pq}\leq0.
\end{equation*}
Moreover, \eqref{sectional-trace-bounds} and
\eqref{F-phi-bounds} imply
\begin{equation*}
 \dot F^{ij}(h^2)_{ij}\leq\frac{n-\ell}{k-\ell}C_F^{k-\ell+1},
 \qquad
 \dot F^{ij}g_{ij}\leq\frac{k}{k-\ell}C_F,
 \qquad
 1\leq\phi\leq C_F.
\end{equation*}
Consequently,
\begin{equation}\label{kappa-upper-ODE}
\begin{aligned}
 0
 &\leq(\partial_t-\mathcal L)f\\
 &\leq
 \frac{(n-\ell)C_F^{k-\ell+1}+kC_F}{k-\ell}f-f^2+C_F\\
 &=
 \frac{(n-\ell)C_F^{k-\ell+1}+kC_F}{k-\ell}C_\kappa
 -C_\kappa^2+C_F
 <0,
\end{aligned}
\end{equation}
a contradiction. Hence $\kappa_n<C_\kappa$ throughout
$[0,T_{\max})$. Together with \eqref{kappa-lower}, this completes the proof.
\end{proof}

We have obtained uniform two-sided bounds for the principal
curvatures. On each time interval in Lemma~\ref{fixed-center},
write $M_t$ as a radial graph $r=r(\theta,t)$ about the fixed
center. Then \eqref{fixed-support-bounds} and \eqref{radial-graph} give
\begin{equation*}
 \frac{r_0}{4}\leq r\leq2R_0,\qquad
 v=\frac{\sinh r}{u}\leq C,\qquad
 |Dr|_\sigma+|D^2r|_\sigma\leq C.
\end{equation*}
Thus the corresponding radial graphs have uniform $C^2$ bounds.
By \eqref{quotient-product} and \eqref{dual-quotient}, $F$ is
concave and inverse-concave. We apply the regularity argument in
\cite{AW}*{Section~5}, as in
\cite{ACW}*{Section~4.3}, to obtain uniform H\"older estimates
for the second derivatives. Parabolic Schauder estimates
\cite{Lieberman} then give all higher-order bounds.
A standard continuation argument yields $T_{\max}=\infty$.

\subsection{Convergence and the inequalities}\label{subsec-convergence}

By \eqref{W-monotonicity} and positivity of $W_k$,
\begin{equation*}
 \int_0^\infty\int_{M_t}E_\ell
       \bigl(F^{k-\ell}-\phi^{k-\ell}\bigr)(F-\phi)\,d\mu_t\,dt
 <\infty.
\end{equation*}
Since $F,\phi\geq1$,
\begin{equation*}
 \bigl(F^{k-\ell}-\phi^{k-\ell}\bigr)(F-\phi)
 =(F-\phi)^2\sum_{j=0}^{k-\ell-1}F^{k-\ell-1-j}\phi^j
 \geq(k-\ell)(F-\phi)^2.
\end{equation*}
Hence there exist times $t_j\to\infty$ such that
\begin{equation}\label{vanishing-speed-variance}
 \int_{M_{t_j}}E_\ell(F-\phi)^2\,d\mu_{t_j}\longrightarrow0.
\end{equation}
The uniform curvature and higher-order estimates, together
with the fixed outer ball and the inner-radius bound, imply
that, after passing to a subsequence and reparametrizing,
$M_{t_j}$ converges smoothly to a closed strictly convex
hypersurface $M_\infty$. By \eqref{F-phi-bounds}, we may also
assume that $\phi(t_j)\to c>0$. Passing to the limit in
\eqref{vanishing-speed-variance} gives
\begin{equation*}
 \int_{M_\infty}E_\ell(F-c)^2\,d\mu_\infty=0.
\end{equation*}
Since $E_\ell>0$, we obtain
$E_k=c^{k-\ell}E_\ell$ on $M_\infty$.
The limit is embedded and strictly convex, so Koh's rigidity
theorem \cite{Koh}*{Theorem~B} implies that it is a geodesic
sphere. Its radius $r_\infty$ is uniquely determined by
$W_\ell(B_{r_\infty})=W_\ell(\Omega_0)$.

To obtain convergence of the entire flow, we apply the
linearization argument in \cite{AW}*{Section~6}, as used in
\cite{ACW}*{Section~4.3}. At a geodesic sphere, symmetry and
homogeneity give $\dot F^i=1/n$. Since $F$ is constant there,
the variations of the weight $E_\ell\,d\mu$ cancel between the
numerator and denominator of \eqref{normalization}.
Consequently, $\delta\phi=\overline{\delta F}$, and the
linearization of the normal graph equation at a sphere of
radius $r_\infty$ is
\begin{equation}\label{linearized-sphere}
 \mathscr L\eta
 =\frac1{n\sinh^2r_\infty}
  \left(\Delta_\sigma\eta+n\eta-n\overline\eta\right),
 \qquad
 \overline\eta=\frac1{|\mathbb S^n|}
                  \int_{\mathbb S^n}\eta\,d\sigma.
\end{equation}
The nearby equilibria are precisely the geodesic spheres by
\cite{Koh}*{Theorem~B}. They form an $(n+2)$-dimensional smooth
manifold whose tangent space is $\ker\mathscr L$, consisting of
the constants and the first spherical harmonics. The zero
eigenvalue is semisimple, and on harmonics of degree $j\geq2$,
\begin{equation*}
 \mathscr L\eta
 =\frac{n-j(j+n-1)}{n\sinh^2r_\infty}\,\eta,
 \qquad n-j(j+n-1)<0.
\end{equation*}
The normal-graph operator is smooth and uniformly parabolic
near the sphere. Thus the normal stability theorem
\cite{PSZ}*{Theorem~3.1}, applied with $A=-\mathscr L$,
gives exponential convergence for solutions starting sufficiently
near $M_\infty$. Since $M_{t_j}\to M_\infty$ smoothly, the entire
flow converges smoothly and exponentially to a single geodesic
sphere, with radius $r_\infty$ by preservation of $W_\ell$.
This completes the proof of Theorem~\ref{main-flow}.

\begin{proof}[Proof of Theorem~\ref{AF}]
First suppose that $\partial\Omega$ has positive sectional curvature.
For $1\leq \ell<k$ with $k+\ell\leq n$, \eqref{W-monotonicity} and
Theorem~\ref{main-flow} give
\begin{equation*}
 W_k(\Omega)\geq W_k(B_{r_\infty})
 =f_k\circ f_\ell^{-1}\bigl(W_\ell(\Omega)\bigr).
\end{equation*}
For $\ell=0$, the same inequality follows from
\cite{ACW}*{Corollary~1.5}.

Now assume $\operatorname{sec}_{\partial\Omega}\geq0$.
Convexity and \eqref{Gauss} give
\begin{equation*}
 \kappa_i\geq0,\qquad \kappa_i\kappa_j\geq1\quad(i\ne j),
 \qquad\text{hence }\kappa_i>0.
\end{equation*}
Let $\Omega_c$ be the domains obtained by Klein contraction.
By Lemma~\ref{Klein-contraction},
\begin{equation*}
 \operatorname{sec}_{\partial\Omega_c}>0,\qquad
 \partial\Omega_c\longrightarrow\partial\Omega
 \quad\text{smoothly as }c\uparrow1.
\end{equation*}
Applying the positive-curvature case and passing to the limit yields
\begin{equation}\label{AF-nonnegative-limit}
\begin{aligned}
 W_k(\Omega)
 &=\lim_{c\uparrow1}W_k(\Omega_c)\\
 &\geq\lim_{c\uparrow1}f_k\circ f_\ell^{-1}\bigl(W_\ell(\Omega_c)\bigr)
 =f_k\circ f_\ell^{-1}\bigl(W_\ell(\Omega)\bigr).
\end{aligned}
\end{equation}

Suppose equality holds. Starting from $\partial\Omega$, use
\eqref{flow}--\eqref{normalization}. For $\ell=0$, this is the
volume-preserving flow with $F=E_k^{1/k}$ and $E_0=1$.
The local theory in Section~\ref{sec-prelim} gives a common
interval $[0,T]$ of smooth embedded strictly convex solutions
from $\partial\Omega_c$ for $c\in[c_0,1]$, where $c_0<1$ is
sufficiently close to one, with
\begin{equation*}
 M_{c,t}\longrightarrow M_t=\partial\Omega_t
 \quad\text{in }C^2\text{ as }c\uparrow1,
 \qquad \Omega_0=\Omega.
\end{equation*}
For $c\in[c_0,1)$, positive sectional curvature is preserved by
Theorem~\ref{preservation} when $\ell\geq1$, and by
\cite{ACW}*{Theorem~3.1} when $\ell=0$. Hence
$\operatorname{sec}_{M_t}\geq0$ on $[0,T]$.
The identities \eqref{W-monotonicity} also hold for $\ell=0$, so
\eqref{AF-nonnegative-limit} gives
\begin{equation*}
\begin{aligned}
 W_k(\Omega_0)
 &\geq W_k(\Omega_t)
 \geq f_k\circ f_\ell^{-1}\bigl(W_\ell(\Omega_t)\bigr)\\
 &=f_k\circ f_\ell^{-1}\bigl(W_\ell(\Omega_0)\bigr)
 =W_k(\Omega_0).
\end{aligned}
\end{equation*}
Thus $W_k(\Omega_t)$ is constant. Evaluating \eqref{W-monotonicity}
at $t=0$, and using $E_\ell>0$, yields
\begin{equation*}
 F\equiv\phi(0),\qquad E_k=\phi(0)^{k-\ell}E_\ell
 \quad\text{on }\partial\Omega.
\end{equation*}
Koh's theorem \cite{Koh}*{Theorem~B} implies that $\Omega$ is
a geodesic ball. Conversely, every geodesic ball attains equality.
\end{proof}

\begin{proof}[Proof of Corollary~\ref{AF-H6}]
Convexity and nonnegative sectional curvature imply strict convexity,
as in the proof of Theorem~\ref{AF}. For $3\leq n\leq5$ and
$0\leq \ell<k\leq n-2$, we have
$k+\ell\leq2n-5\leq n$, so Theorem~\ref{AF} gives these
comparisons. The comparisons with upper index $n-1$ follow
from \cite{HL}*{Theorem~1.3}, and those with upper index $n$
follow from \cite{BGL}. The latter result is also stated
in \cite{GuanLiSurvey}*{Theorem~4.2(2)}. These two endpoint results also give
all comparisons when $n=2$. Their equality statements prove
the asserted rigidity.
\end{proof}

\section{Sharpness of the index range}\label{sec-counterexamples}

We prove Theorem~\ref{prop-section-counterexample} in three steps.
First, we construct a strictly convex local graph in the Klein
model whose minimum sectional curvature is nonnegative, vanishes
only at the origin, and has a negative initial time derivative
there. We then complete the graph to a smooth closed hypersurface
with nonnegative sectional curvature, leaving a neighbourhood of
the origin unchanged. Finally, Klein contraction gives strictly
positive initial sectional curvature, and short-time continuous
dependence yields the required loss of positivity.

\begin{proof}[Proof of Theorem~\ref{prop-section-counterexample}]
\emph{Step 1. Construction of the local graph.}
For a graph $Y=(x,w(x))$ in the Klein ball, let $g=(g_{ij})$ and $h=(h_{ij})$
be its reduced metric and second fundamental forms with respect to the downward
normal, and write $D_E$ for Euclidean differentiation. The
principal curvatures are the eigenvalues of
\begin{equation}\label{eq:counter-graph-matrix}
 \mathcal A=g^{-1/2}hg^{-1/2}.
\end{equation}
By \eqref{Klein-metric-comparison} and
\eqref{Klein-second-form-comparison}, this matrix is a positive
congruence of $D_E^2w$, with coefficients depending smoothly on
$(x,w,D_Ew)$. If $w(o)=D_Ew(o)=0$ at $o=0$, then
\begin{equation*}
 g(o)=I,\qquad D_Eg(o)=0,\qquad
 \mathcal A(o)=D_E^2w(o),\qquad
 \nabla_mh_{ij}(o)=\partial_{mij}w(o).
\end{equation*}

Write $x=(x_1,x_2,y)$, where $y\in\mathbb R^{n-2}$. For $a>1$
close to one, define
\begin{equation}\label{eq:counter-quartic-graph}
\begin{aligned}
 w_a(x)={}&
 \frac12\bigl(a^{-1}x_1^2+ax_2^2+a^2|y|^2\bigr)
 +\frac{x_1^3}{6}-\frac{a^2x_1x_2^2}{2}\\
 &+\frac{a^3}{4}|x|^4
 +\frac1{12}\bigl(a^{-1}x_1^4+ax_2^4\bigr).
\end{aligned}
\end{equation}
Let $\mathcal A_a$ denote the matrix
\eqref{eq:counter-graph-matrix} of this graph. At $o$,
\begin{equation*}
 \mathcal A_a(o)=\operatorname{diag}(a^{-1},a,a^2,\ldots,a^2),
\end{equation*}
and the only nonzero components of $\nabla h$ are
\begin{equation}\label{eq:counter-third-jets}
 \nabla_1h_{11}(o)=1,\qquad
 \nabla_1h_{22}(o)=\nabla_2h_{12}(o)=\nabla_2h_{21}(o)=-a^2.
\end{equation}
The first two principal curvatures are simple at $o$. Hence
$S_a=\kappa_1\kappa_2-1$ is smooth near $o$, with
\begin{equation*}
 S_a(o)=0,\qquad \nabla S_a(o)=0.
\end{equation*}

We compute the quadratic part of $S_a$. The Klein formulas give
\begin{equation}\label{eq:counter-klein-jets}
\begin{aligned}
 g={}&I+|x|^2I+x\otimes x
          +D_Ew_a\otimes D_Ew_a+O_a(|x|^3),\\
 h={}&\left(1+\frac{|x|^2-|D_Ew_a|^2}{2}\right)
             D_E^2w_a+O_a(|x|^3),\\
 g^{-1/2}={}&I-\frac12\left(|x|^2I+x\otimes x
                  +D_Ew_a\otimes D_Ew_a\right)+O_a(|x|^3).
\end{aligned}
\end{equation}
Substitution into \eqref{eq:counter-graph-matrix} yields, for
$i=1,2$,
\begin{equation}\label{eq:counter-matrix-taylor}
\begin{aligned}
 (\mathcal A_a)_{ii}
 ={}&\partial_{ii}w_a
 -\frac{\kappa_i(o)}2
       \sum_{j=1}^n\bigl(1+\kappa_j(o)^2\bigr)x_j^2\\
 &-\kappa_i(o)\bigl(1+\kappa_i(o)^2\bigr)x_i^2
 +O_a(|x|^3).
\end{aligned}
\end{equation}
Moreover,
\begin{equation*}
 (\mathcal A_a)_{12}=-a^2x_2+O_a(|x|^2),\qquad
 (\mathcal A_a)_{ip}=O_a(|x|^2)
 \quad\text{for }i=1,2,\quad p\geq3.
\end{equation*}
For fixed $a>1$, the first two eigenvalues are separated from
the remaining spectrum at $o$. Since the mixed block is of
order two, their product agrees to second order with the
determinant of the leading $2\times2$ block. Using
\eqref{eq:counter-quartic-graph} and
\eqref{eq:counter-matrix-taylor}, we obtain
\begin{equation}\label{eq:counter-section-flat-minimum}
\begin{aligned}
 S_a(x)
 &=(\mathcal A_a)_{11}(\mathcal A_a)_{22}
           -(\mathcal A_a)_{12}^2-1+O_a(|x|^3)\\
 &=(3a^4-1-2a^{-2})x_1^2
           +(a^2-1)(x_2^2+|y|^2)+O_a(|x|^3)\\
 &=(a^2-1)\bigl[(3a^2+3+2a^{-2})x_1^2
                         +x_2^2+|y|^2\bigr]+O_a(|x|^3).
\end{aligned}
\end{equation}
Thus $S_a$ has a strict local minimum at $o$.

We next compute its initial variation under the normal velocity
$(\phi-F)\nu$. Equations \eqref{h-uncommuted} and
\eqref{metric-evolution} give, at $o$,
\begin{equation}\label{eq:counter-section-time-derivative}
 \partial_tS_a(o,0)
 =a\nabla_{11}F+a^{-1}\nabla_{22}F,
\end{equation}
since the zero-order term is
$(F-\phi)(a+a^{-1})(\kappa_1\kappa_2-1)=0$.
In particular, this derivative is independent of $\phi(0)$.

At $a=1$, the expansion \eqref{eq:counter-klein-jets} and the
identity
\begin{equation*}
 D_E^2\left(\frac{|x|^4}{4}\right)
 =|x|^2I+2x\otimes x
\end{equation*}
give
\begin{equation}\label{eq:counter-klein-cancellation}
 \mathcal A_1
 =I+\left(
 \begin{pmatrix}
  x_1+x_1^2&-x_2\\
  -x_2&-x_1+x_2^2
 \end{pmatrix}\oplus0_{n-2}
 \right)+O(|x|^3).
\end{equation}
Consequently, for $m=1,2$,
\begin{equation*}
 \operatorname{tr}\bigl(\partial_m\mathcal A_1(o)\bigr)=0,
 \qquad
 \operatorname{tr}\bigl((\partial_m\mathcal A_1(o))^2\bigr)=2,
 \qquad
 \operatorname{tr}\bigl(\partial_{mm}\mathcal A_1(o)\bigr)=2.
\end{equation*}
For every symmetric matrix $\xi$, differentiation of
$F=(E_k/E_\ell)^{1/(k-\ell)}$ at $I$ yields
\begin{equation}\label{eq:counter-speed-derivatives}
\begin{aligned}
 \dot F^{ij}(I)\xi_{ij}
 &=\frac{\operatorname{tr}\xi}{n},\\
 \ddot F^{ij,pq}(I)\xi_{ij}\xi_{pq}
 &=\frac{k+\ell-1}{n^2(n-1)}
   \bigl((\operatorname{tr}\xi)^2-n\operatorname{tr}(\xi^2)\bigr).
\end{aligned}
\end{equation}
The chain rule therefore gives
\begin{equation}\label{eq:counter-speed-second-derivative}
 \left.\partial_{mm}\bigl(F(\mathcal A_1)\bigr)\right|_o
 =\frac2n-\frac{2(k+\ell-1)}{n(n-1)},
 \qquad m=1,2.
\end{equation}
Since $F$ is smooth near $I$ and the Christoffel symbols at $o$
vanish for every $a$, it follows from
\eqref{eq:counter-section-time-derivative} that
\begin{equation}\label{eq:counter-section-negative-limit}
 \lim_{a\downarrow1}\partial_tS_a(o,0)
 =\frac4n-\frac{4(k+\ell-1)}{n(n-1)}
 =-\frac{4(k+\ell-n)}{n(n-1)}<0.
\end{equation}

By \eqref{eq:counter-section-negative-limit}, choose and fix
$1<a<\sqrt2$ sufficiently close to $1$ that
\begin{equation}\label{eq:counter-local-negative-bound}
 \partial_tS_a(o,0)
 <-\frac{2(k+\ell-n)}{n(n-1)}<0.
\end{equation}
Set $w=w_a$ and $\mathcal A=\mathcal A_a$. By
\eqref{eq:counter-section-flat-minimum}, there is
$0<\varrho<1/16$ such that
\begin{equation}\label{eq:counter-local-bounds}
 0<D_E^2w<2I,\qquad
 S_a(x)\geq\frac{a^2-1}{2}|x|^2
 \quad\text{on }B_\varrho(0).
\end{equation}
We keep $a$ and $\varrho$ fixed in the remaining steps.

\medskip
\noindent\emph{Step 2. Completion to a closed hypersurface.}
We adapt the spherical completion of \cite{AMZ}*{Sections~4--5}.
For $\delta>0$ sufficiently small, let
\begin{equation*}
 b_\delta(x)=\frac14-\delta-\sqrt{\frac1{16}-|x|^2}
\end{equation*}
be the lower cap of the Euclidean sphere of radius $1/4$
centered at $(0,1/4-\delta)$. The sphere is contained in
$B_{1/2}(0)$, and its hyperbolic principal curvatures satisfy
\begin{equation*}
 \kappa_i^b\geq4(1-|Y|^2)^2\geq\frac94
\end{equation*}
by \eqref{Klein-metric-comparison} and
\eqref{Klein-second-form-comparison}. Since
$D_E^2b_\delta\geq4I$, \eqref{eq:counter-local-bounds} gives
\begin{equation*}
 -\delta+|x|^2\leq b_\delta-w\leq-\delta+C|x|^2
 \qquad\text{on }B_\varrho(0).
\end{equation*}
Consequently, $w>b_\delta$ near $o$, whereas $b_\delta>w$
on a collar of $\partial B_\varrho(0)$, provided $\delta$ is
sufficiently small. The relative position of the graphs is
illustrated in Figure~\ref{fig:counter-gluing-sketch}.

\begin{figure}[htbp]
\centering
\begin{tikzpicture}[x=1.5cm,y=1.5cm]
 \draw[->,gray] (-2.10,0) -- (2.35,0) node[right] {$x$};
 \draw[->,gray] (0,-0.34) -- (0,1.28) node[above] {$z$};
 \draw[densely dashed,gray] (-1.80,0) -- (-1.80,1.05);
 \draw[densely dashed,gray] (1.80,0) -- (1.80,1.05);
 \node[below] at (-1.80,-0.04) {$-\varrho$};
 \node[below] at (1.80,-0.04) {$\varrho$};
 \draw[thick,blue!70!black,domain=-1.8:1.8,samples=90]
   plot (\x,{1.82-sqrt(4-\x*\x)});
 \draw[thick,red!75!black,domain=-1.8:1.8,samples=90]
   plot (\x,{0.19*\x*\x+0.01*\x*\x*\x});
 \fill (0,0) circle (0.9pt);
 \node[below left] at (0,0) {$o$};
 \fill[blue!70!black] (0,-0.18) circle (0.8pt);
 \node[left] at (0,-0.18) {$-\delta$};
 \node[blue!70!black,right] at (1.83,0.98) {$b_\delta$};
 \node[red!75!black,right] at (1.83,0.67) {$w$};
\end{tikzpicture}
\caption{Schematic section of the local graph $w$ and the lower
spherical cap $b_\delta$. The cap is drawn as a circular arc.
Near $o$ one has $w>b_\delta$, while near the boundary of the
patch one has $b_\delta>w$.}
\label{fig:counter-gluing-sketch}
\end{figure}
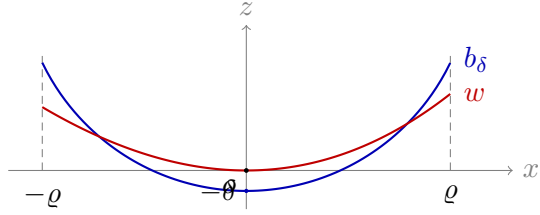

The upper envelope $\max\{w,b_\delta\}$ agrees with $w$
near $o$ and with $b_\delta$ near the boundary, but need not
be smooth along their intersection. We replace it by a smooth
maximum. Choose an even nonnegative function
$\chi\in C_c^\infty((-1,1))$ with $\int_\mathbb R\chi=1$,
and define
\begin{equation}\label{eq:counter-smoothed-positive-part}
 \zeta_\delta(s)
 =\int_\mathbb R (s-\delta^2t)_+\chi(t)\,dt,
 \qquad (q)_+=\max\{q,0\}.
\end{equation}
Evenness of $\chi$ implies
\begin{equation*}
 0\leq\zeta_\delta'\leq1,\qquad
 \zeta_\delta''\geq0,\qquad
 \zeta_\delta(s)=
 \begin{cases}
 0,&s\leq-\delta^2,\\
 s,&s\geq\delta^2.
 \end{cases}
\end{equation*}
Set
\begin{equation*}
 \widetilde w=w+\zeta_\delta(b_\delta-w),\qquad
 \vartheta=\zeta_\delta'(b_\delta-w).
\end{equation*}
Then $\widetilde w=w$ in a neighbourhood of $o$ and
$\widetilde w=b_\delta$ on the boundary collar. In the
transition region $|b_\delta-w|\leq\delta^2$, the preceding
comparison yields $|x|\leq C\sqrt\delta$. Moreover,
\begin{equation*}
 0\leq\widetilde w-w\leq2\delta^2\vartheta,
 \qquad
 |D_E\widetilde w-D_Ew|\leq C\vartheta\sqrt\delta.
\end{equation*}
Since $\zeta_\delta''\geq0$ and
$D_E^2b_\delta-D_E^2w\geq2I$,
\begin{equation*}
 \begin{aligned}
 D_E^2\widetilde w
 &=(1-\vartheta)D_E^2w+\vartheta D_E^2b_\delta
  +\zeta_\delta''(b_\delta-w)
     D_E(b_\delta-w)\otimes D_E(b_\delta-w)\\
 &\geq D_E^2w+2\vartheta I.
 \end{aligned}
\end{equation*}

The Klein formulas express the symmetric Weingarten matrix
of a graph $v$ as
$\mathcal A_v=\mathcal T_vD_E^2v\,\mathcal T_v$,
where $\mathcal T_v$ is positive definite, depends smoothly
on $(x,v,D_Ev)$, and equals $I$ when $(x,v,D_Ev)=(0,0,0)$.
On the transition region,
\begin{equation*}
 \|\mathcal T_{\widetilde w}-\mathcal T_w\|
 \leq C\vartheta\sqrt\delta,
 \qquad
 \mathcal T_{\widetilde w}^{\,2}\geq(1-C\sqrt\delta)I.
\end{equation*}
Together with the uniform bound for $D_E^2w$, this gives
\begin{equation}\label{eq:counter-smoothing-matrix}
 \begin{aligned}
 \widetilde{\mathcal A}
 &=\mathcal T_{\widetilde w}D_E^2\widetilde w\,
   \mathcal T_{\widetilde w}\geq\mathcal A+(2-C\sqrt\delta)\vartheta I
 \geq\mathcal A+\vartheta I
 \end{aligned}
\end{equation}
for sufficiently small $\delta$, where
$\widetilde{\mathcal A}=\mathcal A_{\widetilde w}$ and
$\mathcal A=\mathcal A_w$. Consequently, in the transition region,
\begin{equation*}
 \widetilde\kappa_i\geq\kappa_i+\vartheta,
 \qquad
 \widetilde\kappa_1\widetilde\kappa_2
 \geq(\kappa_1+\vartheta)(\kappa_2+\vartheta)
 \geq(1+\vartheta)^2,
\end{equation*}
where we used $\kappa_1\kappa_2\geq1$ and
$\kappa_1+\kappa_2\geq2$. The inequality is strict when
$\vartheta>0$. Outside this region the smoothed graph agrees
with either $w$ or $b_\delta$.
It is therefore strictly convex and has positive sectional
curvature away from $o$.

Replacing the lower spherical cap over $B_\varrho(0)$ by the
graph of $\widetilde w$ produces a smooth closed embedded
hypersurface $\Sigma$. Indeed, the graphs agree on the boundary
collar, and the modified lower cap remains below the upper cap
$b_\delta^+(x)=\frac14-\delta+\sqrt{\frac1{16}-|x|^2}$,
since
\begin{equation*}
 0\leq\widetilde w\leq\max\{w,b_\delta\}+\delta^2
 \leq4\varrho^2+\delta^2<\frac14<b_\delta^+(x)
 \qquad\text{on }B_\varrho(0).
\end{equation*}
Here we used $0\leq w\leq|x|^2$ and
$b_\delta\leq4|x|^2$ on $B_\varrho(0)$. Decreasing
$\delta$ if necessary, we have $\Sigma\subset B_{1/2}(0)$.
Let $\widehat M$ be its inverse Klein image. Then
\begin{equation}\label{eq:counter-closed-data}
 \kappa_i(\widehat M)>0,\qquad
 \operatorname{sec}_{\widehat M}\geq0,\qquad
 \operatorname{sec}_{\widehat M}>0
 \quad\text{on }\widehat M\setminus\{o\}.
\end{equation}
Since $\widehat M$ agrees with the local graph near $o$ and
\eqref{eq:counter-section-time-derivative} is independent of
$\phi(0)$, the initial variation at $o$ retains the bound
\eqref{eq:counter-local-negative-bound}.

\medskip
\noindent\emph{Step 3. Strict approximation and loss of sectional curvature.}
For the hypersurface $\Sigma$ obtained in Step~2, define
\begin{equation}\label{eq:counter-contracted-family}
 M_c=\left\{
 \frac{(1,cY)}{\sqrt{1-c^2|Y|^2}}
 \ \middle|\ Y\in\Sigma
 \right\},\qquad 0<c\leq1.
\end{equation}
Then $M_1=\widehat M$. By Lemma~\ref{Klein-contraction},
\begin{equation}\label{eq:counter-contraction}
 \operatorname{sec}_{M_c}\geq c^{-2}-1>0\quad(0<c<1),
 \qquad M_c\longrightarrow\widehat M
 \quad\text{smoothly as }c\uparrow1.
\end{equation}
The principal curvatures of $\widehat M$ lie in a compact subset
of $\Gamma_+$. The local theory in Section~\ref{sec-prelim}
therefore gives a common interval $[0,T_0]$ of smooth embedded
strictly convex solutions for $c\leq1$ sufficiently close to one,
with continuous dependence in $C^2$.

Parametrize these solutions over $\widehat M$ as $X_c(p,t)$,
with $X_1(p,0)=p$. Denote the ordered principal curvatures at
$X_c(p,t)$ by $\kappa_{1,c}(p,t),\ldots,\kappa_{n,c}(p,t)$,
and set
\begin{equation*}
 S_c(p,t)=\kappa_{1,c}(p,t)\kappa_{2,c}(p,t)-1.
\end{equation*}
Since $\nabla S_1(o,0)=0$, reparametrization does not change
its initial variation. By \eqref{eq:counter-local-negative-bound}
and the agreement of $\widehat M$ with the local graph near $o$,
\begin{equation*}
 S_1(o,0)=0,\qquad
 \partial_tS_1(o,0)<-\frac{2(k+\ell-n)}{n(n-1)}.
\end{equation*}
The first two principal curvatures are simple at $(o,0)$.
By continuity of this derivative, after decreasing $T_0$ and
integrating in time, we obtain
\begin{equation*}
 S_1(o,t)\leq-\frac{k+\ell-n}{n(n-1)}\,t,
 \qquad 0\leq t\leq T_0.
\end{equation*}

Given $T>0$, fix $0<t_*<\min\{T,T_0\}$. Continuous dependence
in $C^2$ implies
\begin{equation*}
 S_c(o,t_*)\longrightarrow S_1(o,t_*)
 \quad\text{as }c\uparrow1.
\end{equation*}
Thus, for $c<1$ sufficiently close to one,
\begin{equation*}
 S_c(o,t_*)<-\frac{k+\ell-n}{2n(n-1)}\,t_*<0.
\end{equation*}
Together with \eqref{eq:counter-contraction}, this gives a
smooth closed initial hypersurface with positive sectional
curvature whose solution develops negative sectional curvature
before time $T$, while remaining smooth and strictly convex.
Since $T>0$ is arbitrary, the theorem follows.
\end{proof}

\section{A paired concavity inequality}\label{sec-algebra}

We prove the estimate used in
Lemma~\ref{lem:third-order-nonnegative}. We first reduce the
quadratic form to a single derivative estimate. The remaining
subsections establish that estimate using elementary symmetric
polynomials. Throughout this section, $q_i=\partial_iq$ and
$q_{ij}=\partial_i\partial_jq$ denote derivatives in the
component variables.

    \begin{proposition}\label{key-lemma}
    Let $n\geq3$ and $1\leq s<r\leq n$, and set
    \begin{equation*}
     q=\left(\frac{\sigma_r}{\sigma_s}\right)^{1/(r-s)}.
    \end{equation*}
    The inequality
    \begin{equation}\label{paired-Jacobi}
     \mathcal J[q](\lambda,\xi)
     =-D^2q(\lambda)[\xi,\xi]
      +2\sum_{i=2}^n\frac{q_i(\lambda)}{\lambda_1-\lambda_i}\xi_i^2
      -\frac{2q_1(\lambda)}{\lambda_1-\lambda_2}\xi_1^2\geq0
    \end{equation}
    holds for every
    \begin{equation*}
     \lambda_1>\lambda_2\geq\cdots\geq\lambda_n>0
    \end{equation*}
    and $\xi\in \mathbb{R}^n$ satisfying the paired constraint 
    \begin{align}\label{pair}
 \lambda_2\xi_1+\lambda_1\xi_2=0,
    \end{align}
    if and only if $r+s\geq n$. In this range, equality holds if
    and only if $\xi=0$.
    \end{proposition}
    
    The adjacent quotients are concave on $\Gamma_+$ by
    \cite{HuiskenSinestrari1999}*{Theorem~2.5}. Their geometric mean
    \begin{equation}\label{alg-quotient-concavity}
     q=\left(\prod_{j=s+1}^r\frac{\sigma_j}{\sigma_{j-1}}
         \right)^{1/(r-s)}
    \end{equation}
    is therefore concave, and Newton's inequalities give $q_i>0$.
    For related concavity estimates, see
    \cite{LiWu}*{Theorem~1.1}, \cite{DongZhang}*{Lemma~1.1}, and
    \cite{Tsai}*{Lemma~1.1}.
    
    \subsection{Reduction of the quadratic form}\label{subsec-alg-main}
    
    By homogeneity, we may normalize
    \begin{equation*}
     \lambda=(a,1,z),\qquad a>1,\qquad
     z=(z_3,\ldots,z_n)\in(0,1]^{n-2}.
    \end{equation*}
    The constraint becomes $\xi_2=-\xi_1/a$, and hence
    \begin{equation}\label{alg-normalized-form}
     \mathcal J[q]
     =-D^2q[\xi,\xi]
      +\sum_{i=3}^n\frac{2q_i}{a-z_i}\xi_i^2
      -\frac2{a-1}\left(q_1-\frac{q_2}{a^2}\right)\xi_1^2.
    \end{equation}
    The first two terms are nonnegative.
    \begin{proof}[Proof of Proposition~\ref{key-lemma}]
    For sufficiency, we assume first that $r+s\geq n$. Define
    \begin{equation*}
     v=\left(a,-a\frac{q_1}{q_2},0,\ldots,0\right),\qquad
     c_i=-D^2q[v,\mathbf e_i],\qquad 1\leq i\leq n,
    \end{equation*}
    where $\mathbf e_i$ is the $i$th coordinate vector.  The paired constraint \eqref{pair} gives
    \begin{equation}\label{pair-qx}
     \left(c_1-\frac{c_2}{a}\right)\xi_1
     =-D^2q[v,\xi]-\sum_{i=3}^n c_i\xi_i.
    \end{equation} For any real $\tau$, writing $\xi=\tau v+(\xi-\tau v)$ and using \eqref{pair-qx}, we have
\begin{equation*}
\begin{aligned}
&-D^2q[\xi,\xi]
 +\sum_{i=3}^n\frac{2q_i}{a-z_i}\xi_i^2\\
={}&-D^2q[\xi-\tau v,\xi-\tau v]
 +\sum_{i=3}^n\frac{2q_i}{a-z_i}
 \left(\xi_i+\tau\frac{a-z_i}{2q_i}c_i\right)^2\\
&\quad
 -2\tau D^2q[v,\xi]
 -2\tau\sum_{i=3}^n c_i\xi_i
 -\tau^2\left(
 -D^2q[v,v]
 +\sum_{i=3}^n\frac{a-z_i}{2q_i}c_i^2
 \right)\\
\geq{}&
 2\tau\left(c_1-\frac{c_2}{a}\right)\xi_1
 -\tau^2\left(
 -D^2q[v,v]
 +\sum_{i=3}^n\frac{a-z_i}{2q_i}c_i^2
 \right).
\end{aligned}
\end{equation*}
where we used the concavity of $q$ and discarded the
nonnegative square terms. Since $\tau$ is arbitrary, we obtain
\begin{equation}\label{alg-direct-sum-CS}
-D^2q[\xi,\xi]
+\sum_{i=3}^n\frac{2q_i}{a-z_i}\xi_i^2
\geq
\frac{\left(c_1-\frac{c_2}{a}\right)^2\xi_1^2}
{-D^2q[v,v]
 +\displaystyle\sum_{i=3}^n\frac{a-z_i}{2q_i}c_i^2}
\end{equation}
To control the last term of \eqref{alg-normalized-form}, in view of \eqref{alg-direct-sum-CS}, we establish the following lemma in
    Subsection~\ref{subsec-alg-completion}. 
    \begin{lemma}\label{alg-derivative-estimate}
    Suppose $r+s\geq n$. At $\lambda=(a,1,z)$, we have
    $-D^2q[v,v]>0$, and
    \begin{equation}\label{alg-strict-schur-comparison}
     \left(c_1-\frac{c_2}{a}\right)^2
     >\frac2{a-1}\left(q_1-\frac{q_2}{a^2}\right)
       \left[-D^2q[v,v]
          +\sum_{i=3}^n\frac{a-z_i}{2q_i}c_i^2\right].
    \end{equation}
    \end{lemma}
If $\xi_1=0$, then \eqref{pair} gives $\xi_2=0$, and
\eqref{alg-normalized-form} is strictly positive for every
nonzero $\xi$, since $q_i>0$.
Thus $\mathcal J[q](\lambda,\xi)\ge0$, with equality
if and only if $\xi=0$.

For necessity, suppose that $r+s<n$ and that
\eqref{paired-Jacobi} holds for all admissible
$(\lambda,\xi)$. We prove this by contradiction. At $\boldsymbol1=(1,\ldots,1)$, we have
\begin{equation}\label{alg-umbilic-derivatives}
 q_i=\frac{q(\boldsymbol1)}n,\qquad
 q_{ii}=-\frac{q(\boldsymbol1)(r+s-1)}{n^2},\qquad
 q_{ij}=\frac{q(\boldsymbol1)(r+s-1)}{n^2(n-1)}\quad(i\ne j).
\end{equation}
Taking $\lambda(a)=(a,1,\ldots,1)$ and
$\xi(a)=(a,-1,0,\ldots,0)$ gives
\begin{equation*}
 \mathcal J[q](\lambda(a),\xi(a))
 =-D^2q[\xi(a),\xi(a)]
  -2\frac{a^2q_1(\lambda(a))-q_2(\lambda(a))}{a-1}.
\end{equation*}
By \eqref{alg-umbilic-derivatives}, the two terms tend to
$2q(\boldsymbol1)(r+s-1)/(n(n-1))$ and
$-2q(\boldsymbol1)(2n-r-s-1)/(n(n-1))$, respectively. Thus
\begin{equation}\label{alg-necessary-range}
 \lim_{a\downarrow1}\mathcal J[q](\lambda(a),\xi(a))
 =\frac{4q(\boldsymbol1)(r+s-n)}{n(n-1)}<0.
\end{equation}
Hence $\mathcal J[q](\lambda(a),\xi(a))<0$ for $a>1$
sufficiently close to $1$, a contradiction.
\end{proof}

\subsection{Symmetric polynomial estimates}\label{subsec-alg-polynomials}

For a positive vector $z$ of length $N$, use $\sigma_j(z)$ with
$\sigma_0=1$ and $\sigma_j=0$ outside $0\leq j\leq N$. For
$1\leq s<r\leq N+2$, set
\begin{equation}\label{alg-ratios}
\begin{aligned}
 P&=\sigma_{r-2}\sigma_{s-1}-\sigma_{r-1}\sigma_{s-2},\\
 Q&=\sigma_{r-2}\sigma_s-\sigma_r\sigma_{s-2},
 &R&=\sigma_{r-1}\sigma_s-\sigma_r\sigma_{s-1},\\
 t&=Q/P,&u&=R/P.
\end{aligned}
\end{equation}
All symmetric polynomials in these expressions are evaluated at
$z$. Newton's inequalities give $P>0$ and $Q,R\geq0$.
We write $t_i=\partial_it$, $u_i=\partial_iu$, and $z|i$ for
the vector obtained by deleting $z_i$.

\begin{lemma}\label{alg-coordinate-bounds}
For every positive $z$,
\begin{equation}\label{alg-coordinate-derivatives}
 0\leq t_i\leq1,\qquad
 0\leq u_i\leq
 \begin{cases}
 t,&r-s\geq2,\\
 2t,&r-s=1.
 \end{cases}
\end{equation}
In particular, if $N=n-2$, $r+s\geq n$, and $z\in(0,1]^N$, then
\begin{equation}\label{alg-ratio-size}
\begin{aligned}
 0&\leq u\leq\frac{(n-r)(n-s)}{rs}\leq1,
\end{aligned}
\end{equation}
and if $r=s+1$, we have 
 $0\leq t\leq\frac{2(n-s-1)}{s+1}<2$.
\end{lemma}

\begin{proof}
We use the minors
\begin{equation*}
 \Delta_{p,q}=\sigma_{p-1}\sigma_q-\sigma_p\sigma_{q-1},\qquad
 \Delta_{p,q,j}=
 \det\begin{pmatrix}
 \sigma_p&\sigma_{p-1}&\sigma_{p-2}\\
 \sigma_q&\sigma_{q-1}&\sigma_{q-2}\\
 \sigma_j&\sigma_{j-1}&\sigma_{j-2}
 \end{pmatrix}.
\end{equation*}
The second minor is nonnegative for $0\leq p\leq q\leq j$.
Indeed, the finite matrix $(\sigma_{p-q})_{0\leq p,q\leq L}$,
with $L\geq\max\{j,2\}$, factors as
\begin{equation*}
 (\sigma_{p-q})_{p,q}
 =\prod_h(\delta_{pq}+z_h\delta_{p,q+1})_{p,q}.
\end{equation*}
Each bidiagonal factor has nonnegative minors, since a nonzero
permutation term must preserve the row and column order.
Cauchy--Binet gives the assertion for the product.

Fix $i$ and put $\tau=z_i$. In the remainder of this proof,
coefficients and minors are evaluated on $x=z|i$. The insertion
identity $\sigma_j(z)=\sigma_j(x)+\tau\sigma_{j-1}(x)$ gives
\begin{equation}\label{alg-insertion}
\begin{aligned}
 P(\tau)&=\alpha+\beta\tau+\gamma\tau^2,\\
 Q(\tau)&=\eta+\theta\tau+\beta\tau^2,
 &R(\tau)&=\delta+\eta\tau+\alpha\tau^2,\\
 \alpha&=\Delta_{r-1,s-1},&\gamma&=\Delta_{r-2,s-2},
 \qquad\delta=\Delta_{r,s},\\
 \beta&=\Delta_{r-2,s-1}+\Delta_{r-1,s-2},
 &\eta&=\Delta_{r-1,s}+\Delta_{r,s-1},\\
 \theta&=2\alpha+\Delta_{r-2,s}+\Delta_{r,s-2}.
\end{aligned}
\end{equation}
Set $\Delta_1=\Delta_{s-1,s,r-1}$ and
$\Delta_2=\Delta_{s-1,r-1,r}$. Expansion yields
\begin{equation}\label{alg-minor-identities}
\begin{aligned}
 \alpha\theta-\beta\eta
 &=\alpha^2-\gamma\delta
  =\sigma_{s-1}\Delta_2+\sigma_{r-1}\Delta_1\geq0,\\
 \alpha\beta-\gamma\eta
 &=\sigma_{r-2}\Delta_1+\sigma_{s-2}\Delta_2\geq0,\\
 \beta^2-\gamma\theta
 &=\sigma_{r-3}\Delta_1+\sigma_{s-3}\Delta_2\geq0,\\
 \alpha\eta-\beta\delta
 &=\sigma_{s-1}\Delta_{s,r-1,r}
   +\sigma_{r-1}\Delta_{s-1,s,r}\geq0.
\end{aligned}
\end{equation}
Primes denote derivatives in $\tau$. Consequently,
\begin{equation}\label{alg-derivative-numerators}
\begin{aligned}
 PQ'-QP'&=(\alpha^2-\gamma\delta)
      +2(\alpha\beta-\gamma\eta)\tau
      +(\beta^2-\gamma\theta)\tau^2\geq0,\\
 PR'-RP'&=(\alpha\eta-\beta\delta)
      +2(\alpha^2-\gamma\delta)\tau
      +(\alpha\beta-\gamma\eta)\tau^2\geq0.
\end{aligned}
\end{equation}
These prove $t_i,u_i\geq0$. The upper bounds follow from
\begin{equation*}
 P^2(1-t_i)=\gamma(R+\tau Q+\tau^2P)\geq0
\end{equation*}
and
\begin{equation*}
\begin{aligned}
 P^2(t-u_i)
 &=P[(\theta-2\alpha)\tau+\beta\tau^2]
    +R(\beta+2\gamma\tau)\geq0\quad(r-s\geq2),\\
 P^2(2t-u_i)
 &=P[\eta+2(\theta-\alpha)\tau+2\beta\tau^2]
    +R(\beta+2\gamma\tau)\geq0\quad(r-s=1).
\end{aligned}
\end{equation*}
Here $\theta\geq2\alpha$ in the first case and
$\theta\geq\alpha$ in the second. Finally, evaluation at
$z=\boldsymbol1$ and coordinatewise monotonicity give
\eqref{alg-ratio-size}.
\end{proof}

\begin{lemma}\label{alg-weighted-tail-bounds}
Suppose $N=n-2$, $r+s\geq n$, and $z\in(0,1]^N$. Let
$z^{(i)}$ be obtained from $z$ by replacing $z_i$ by $1$, and put
\begin{equation}\label{alg-weights}
 \omega_i=\frac{P(z)}{P(z^{(i)})}.
\end{equation}
Then
\begin{equation}\label{alg-weighted-bounds}
 \sum_i\omega_it_i\leq t+2,\qquad
 \sum_i\omega_iu_i\leq4.
\end{equation}
\end{lemma}

\begin{proof}
We first assume $n=r+s$. A monomial of $P$ with
$J$ exponents equal to two has $n-3-2J$ exponents equal to one
and $J+1$ zero exponents. Its coefficient is
\begin{equation}\label{alg-monomial-coefficients}
 p_J=\binom{n-3-2J}{r-2-J}-\binom{n-3-2J}{r-1-J}
 =\frac{(r-s)(n-3-2J)!}{(r-1-J)!(s-1-J)!}>0,
 \quad 0\leq J\leq s-1.
\end{equation}
Distribute each monomial equally among its zero coordinates,
\begin{equation}\label{alg-monomial-allocation}
 C_i=\sum_{m_i=0}\frac{p_Jz^{\mathbf m}}{J+1}>0,
 \qquad \sum_iC_i=P.
\end{equation}
Here $J$ is the number of exponents equal to two in
$\mathbf m$. The two binomial coefficients in
\eqref{alg-monomial-coefficients} count assignments to the
first factor of each product defining $P$.
Since $t$ is homogeneous of degree one, the desired bounds follow
from the coordinate estimates
\begin{equation*}
 (\omega_i-z_i)t_i\leq\frac{2C_i}{P},\qquad
 \omega_i u_i\leq\frac{4C_i}{P}.
\end{equation*}
Indeed, summation uses $\sum_i z_it_i=t$ and $\sum_iC_i=P$.

Fix $i$, write $x=z|i$ and $g=C_i$, and use the coefficients
and minors in \eqref{alg-insertion}. On $x$, the exponent vectors below
have $J$ double exponents, $J$ zero exponents, and
$n-3-2J$ single exponents. Counting gives
\begin{equation}\label{alg-allocation-identities}
\begin{aligned}
 g&=\sum_{\mathbf m}\frac{p_Jx^{\mathbf m}}{J+1},
 &\alpha&=\sum_{\mathbf m}p_Jx^{\mathbf m},\\
 \gamma&=\sum_{\mathbf m}\frac{p_Jx^{\mathbf m}}{J+1}
                   \sum_{m_h=2}x_h^{-2},
 &\delta&=\sum_{\mathbf m}\frac{p_Jx^{\mathbf m}}{J+1}
                   \sum_{m_h=0}x_h^2,\\
 \beta&=\sum_{\mathbf m}\frac{p_Jx^{\mathbf m}}{J+1}
                   \sum_{m_h=1}x_h^{-1},
 &\eta&=\sum_{\mathbf m}\frac{p_Jx^{\mathbf m}}{J+1}
                   \sum_{m_h=1}x_h.
\end{aligned}
\end{equation}
For the last four identities, change one exponent by
$2\to0$, $0\to2$, $1\to0$, or $1\to2$, respectively.
Each resulting monomial has $J+1$ reverse choices, cancelling
the divisor. In particular,
\begin{equation}\label{alg-allocation-bounds}
 \delta\leq\alpha-g,\qquad \eta\leq\beta,\qquad
 \alpha\geq g\geq\alpha/s>0.
\end{equation}
We also claim
\begin{equation}\label{alg-product-bounds}
 \gamma(\delta+g)\geq(\alpha-g)^2,\qquad
 \alpha\beta-\gamma\eta\leq g(\beta+2\gamma).
\end{equation}

For the first estimate, group the monomials by $J$ and
their single-exponent coordinates. On the remaining $2J$
coordinates set $y_h=x_h^2$. Each group contributes
\begin{equation*}
 c\sigma_J(y),\quad cJ\sigma_J(y),\quad
 c(J+1)\sigma_{J-1}(y),\quad c(J+1)\sigma_{J+1}(y)
\end{equation*}
to $g$, $\alpha-g$, $\gamma$, and $\delta$, respectively,
where $c>0$. For $J\geq1$, monotonicity and degree-one
homogeneity of $\sigma_J/\sigma_{J-1}$, together with
$0<y_h\leq1$, give
\begin{equation*}
 \sum_h y_h^2\partial_h\log\frac{\sigma_J}{\sigma_{J-1}}
 \leq1.
\end{equation*}
The identity
\begin{equation}\label{alg-extension-count}
 \sum_h y_h^2\partial_h\sigma_j
 =\sigma_1\sigma_j-(j+1)\sigma_{j+1}
\end{equation}
then gives
\begin{equation*}
 (J+1)\sigma_{J-1}
       [(J+1)\sigma_{J+1}+\sigma_J]\geq J(J+1)\sigma_J^2.
\end{equation*}
For $J=0$, the contributions to $\alpha-g$, $\gamma$, and
$\delta$ vanish. Cauchy--Schwarz over the groups proves the
first estimate.

For the second estimate, if $s=1$, then $g=\alpha$ and
$\gamma=0$, so equality holds. Suppose $s\geq2$. Then
$\gamma>0$, and Lemma~\ref{alg-coordinate-bounds}, applied with
indices $(r-1,s-1)$ on $x$, shows that $\alpha/\gamma$ is increasing.
Since it has degree two,
\begin{equation*}
 \sum_hx_h^2\partial_h\log\frac{\alpha}{\gamma}\leq2.
\end{equation*}
Using \eqref{alg-extension-count} and expanding the minors,
we obtain the identity
\begin{equation*}
\begin{aligned}
 s(\alpha\beta-\gamma\eta)
 &=\alpha\beta
   +\alpha\gamma\sum_hx_h^2\partial_h\log\frac{\alpha}{\gamma}
   -(r-s)\sigma_{s-2}\Delta_2\\
 &\leq\alpha(\beta+2\gamma)
 \leq sg(\beta+2\gamma).
\end{aligned}
\end{equation*}
This proves \eqref{alg-product-bounds}.

We now estimate the derivatives in the fixed coordinate
$\tau=z_i$. The insertion identities give
\begin{equation*}
\begin{aligned}
 PQ'-QP'
 &=(\alpha^2-\gamma\delta)
       \left(1-\frac{\gamma\tau^2}{\alpha}\right)
   +(\alpha\beta-\gamma\eta)
       \left(2\tau+\frac{\beta\tau^2}{\alpha}\right),\\
 \omega_i-\tau&=\frac{(1-\tau)(\alpha-\gamma\tau)}{P(1)}.
\end{aligned}
\end{equation*}
Since $PQ'-QP'\geq0$ and $0<\tau\leq1$, we have
\begin{equation*}
 (1-\tau)(\alpha-\gamma\tau)(PQ'-QP')
 \leq\alpha(1-\tau)(PQ'-QP').
\end{equation*}
The two estimates in \eqref{alg-product-bounds} give
\begin{equation*}
 \alpha^2-\gamma\delta\leq g(2\alpha+\gamma-g),\qquad
 \alpha\beta-\gamma\eta\leq g(\beta+2\gamma).
\end{equation*}
Using $1-\gamma\tau^2/\alpha\leq1$ in the expression for
$PQ'-QP'$, followed by expansion and omission of nonnegative
terms, yields
\begin{equation*}
\begin{aligned}
 &2gP(1)P(\tau)
       -(1-\tau)(\alpha-\gamma\tau)(PQ'-QP')\\
 \geq&
 g[2\alpha^2\tau+\alpha g(1-\tau)
          +\alpha\gamma(1-3\tau+6\tau^2)]\\
          >&0.
\end{aligned}
\end{equation*}
Division by $P(1)P(\tau)^2$ proves
$(\omega_i-z_i)t_i\leq2C_i/P$.

For the second coordinate estimate,
\eqref{alg-derivative-numerators} shows that $PR'-RP'$ has
nonnegative coefficients. It is therefore enough to bound its
value at $\tau=1$. Direct expansion gives
\begin{equation}\label{alg-u-endpoint-difference}
\begin{aligned}
 &\alpha\bigl[4gP(1)-(PR'-RP')(1)\bigr]\\
 =&(\alpha-g-\delta)
       \bigl[g(\beta+2\gamma)-(\alpha\beta-\gamma\eta)\bigr]\\
 &+(2\alpha+\eta)
       \bigl[\gamma(\delta+g)-(\alpha-g)^2\bigr]\\
 &+2g\bigl[\alpha g+\gamma(\delta+g)\bigr]
       +g(\delta+2g)\eta\\
 &+g(2\alpha+\delta+g)(\beta-\eta)>0.
\end{aligned}
\end{equation}
All terms are nonnegative by \eqref{alg-allocation-bounds}
and \eqref{alg-product-bounds}, and $2\alpha g^2>0$.
Consequently,
\begin{equation*}
 \omega_i u_i
 =\frac{PR'-RP'}{P(\tau)P(1)}\leq\frac{4C_i}{P(\tau)}.
\end{equation*}
This proves both coordinate estimates and hence
\eqref{alg-weighted-bounds} when $n=r+s$, including $s=1$.

Finally, suppose $r+s>n$. Append $r+s-n$ coordinates equal
to $\varepsilon\in(0,1]$ to $z$ and apply the bounds in
dimension $r+s$. All added summands are nonnegative and may
be discarded. As $\varepsilon\downarrow0$, the polynomials
and their derivatives in the original coordinates converge.
The limiting denominators are positive by Newton's
inequalities. This includes $r=n$, for which
$P=\sigma_{n-2}\sigma_{s-1}>0$ on the original vector.
The same applies to $P(z^{(i)})$. Passing to the limit proves
\eqref{alg-weighted-bounds} in every stated dimension.
\end{proof}

\subsection{Proof of the derivative estimate}\label{subsec-alg-completion}

\begin{proof}[Proof of Lemma~\ref{alg-derivative-estimate}]
Let $\lambda=(a,1,z)$, where $a>1$, and assume $r+s\geq n$. Expanding the symmetric
polynomials in the first two variables gives
\begin{equation*}
 \sigma_s\partial_1\sigma_r-\sigma_r\partial_1\sigma_s
 =P+Q+R,
 \qquad
 \frac{q_1}{q_2}
 =\frac{1+t+u}
        {a^2+at+u}.
\end{equation*}
For brevity, set
\begin{equation}\label{alg-scalar-quantities}
    \mathsf U=at+(a+1)u,\qquad
    \mathsf S=1+t+u,\qquad
    \mathsf T=a+1+t,\qquad
    \mathsf C=a^2+at+u,
\end{equation}
and
\begin{equation*}
 \mathsf H=a\mathsf S(2a+t)+\mathsf C(t+2),
 \qquad
 \mathsf K=a(a+1)t+a(3a-1)+(a+1)u.
\end{equation*}
Thus $q_1/q_2=\mathsf S/\mathsf C$. Since $Dq[v]=0$,
\[
 c_i=-a q_2\partial_i(q_1/q_2).
\]
Differentiating the quotient before setting $(\lambda_1,\lambda_2)=(a,1)$
yields
\begin{equation}\label{alg-auxiliary-coefficients}
 \begin{aligned}
 c_1&=\frac{a q_2\mathsf S(2a+t)}{\mathsf C^2},
 &c_2&=-\frac{a q_2(t+2)}{\mathsf C},\\
 c_i&=-\frac{a q_2(a-1)}{\mathsf C^2}
       \bigl[(a-u)t_i+\mathsf T u_i\bigr],
 &&i\geq3.
 \end{aligned}
\end{equation}
Consequently,
\begin{equation}\label{alg-auxiliary-identities}
 \begin{aligned}
 c_1-\frac{c_2}{a}
 &=\frac{q_2\mathsf H}{\mathsf C^2},\\
 -D^2q[v,v]
 &=\frac{2a^2q_2\mathsf S\mathsf T}{\mathsf C^2}>0,\\
 \frac2{a-1}\left(q_1-\frac{q_2}{a^2}\right)
 &=\frac{2q_2\mathsf U}{a^2\mathsf C}.
 \end{aligned}
\end{equation}

For $i\geq3$, the numerator
$\sigma_s\partial_i\sigma_r-\sigma_r\partial_i\sigma_s$
does not depend on $z_i$. Replacing $z_i$ by $1$ and using symmetry,
together with the monotonicity of $t$ and $u$, gives
\begin{equation*}
 \frac{q_i}{q_2}
 =\frac{P(z^{(i)})
 [a^2+a\,t(z^{(i)})+u(z^{(i)})]}{P(z)\mathsf C}
 \geq\frac1{\omega_i}.
\end{equation*}
It follows that
\begin{equation}\label{alg-tail-correction}
 \sum_{i=3}^n\frac{a-z_i}{2q_i}c_i^2
 \leq\frac{a^3q_2(a-1)^2}{2\mathsf C^4}
 \sum_{i=3}^n\omega_i
 \bigl[(a-u)t_i+\mathsf T u_i\bigr]^2.
\end{equation}
By Lemmas~\ref{alg-coordinate-bounds} and
\ref{alg-weighted-tail-bounds},
\begin{equation}\label{alg-tail-square-bound}
 \sum_{i=3}^n\omega_i
 \bigl[(a-u)t_i+\mathsf T u_i\bigr]^2
 \leq
 (a-u)^2(t+2)+4\chi\mathsf T^2t+8(a-u)\mathsf T,
\end{equation}
where $\chi=1$ if $r-s\geq2$ and $\chi=2$ if $r-s=1$.
Indeed, use $t_i^2\leq t_i$, $u_i^2\leq\chi t u_i$,
and $t_i u_i\leq u_i$. In the second case, $t<2$.
Also $0\leq u\leq1$.

It remains to verify the scalar inequality
\begin{equation}\label{alg-final-scalar-bound}
 \mathsf C\bigl(\mathsf H^2-4\mathsf U\mathsf S\mathsf T\mathsf C\bigr)
 >a\mathsf U(a-1)^2
 \sum_{i=3}^n\omega_i
 \bigl[(a-u)t_i+\mathsf T u_i\bigr]^2.
\end{equation}
The key factorization is
\begin{equation}\label{alg-scalar-factorization}
 \mathsf H^2-4\mathsf U\mathsf S\mathsf T\mathsf C
 =4(a-u)\mathsf S\mathsf K\mathsf T
 +(a-1)^2(a-u)^2(t+2)^2.
\end{equation}
Moreover,
\begin{equation*}
 (t+2)\mathsf C-a\mathsf U
 =at^2+2at+(2-u)a^2-ua+u(t+2)>0,
\end{equation*}
since $(2-u)a^2-ua\geq a^2-a>0$. Thus the second term in
\eqref{alg-scalar-factorization}, after multiplication by $\mathsf C$,
strictly dominates the part of the right-hand side of
\eqref{alg-final-scalar-bound} containing $(a-u)^2(t+2)$.
In view of \eqref{alg-tail-square-bound}, it is therefore enough to show
\begin{equation}\label{alg-scalar-reduced}
 \mathsf C(a-u)\mathsf S\mathsf K
 >a\mathsf U(a-1)^2
 \bigl[\chi\mathsf Tt+2(a-u)\bigr].
\end{equation}

If $\chi=1$, then
\[
 \mathsf U\leq a\mathsf S,\quad a-u\geq a-1,\quad
 \mathsf C\geq a(a+t),\quad
 \mathsf K\geq a\bigl[(a+1)t+3a-1\bigr].
\]
Hence \eqref{alg-scalar-reduced} follows from
\begin{equation*}
 \begin{aligned}
 &(a+t)\bigl[(a+1)t+3a-1\bigr]
 -(a-1)\mathsf Tt-2(a-1)^2\\
 =&2t^2+4at+a^2+3a-2>0.
 \end{aligned}
\end{equation*}

If $\chi=2$, then $0\leq t\leq2$, and
\begin{equation*}
 \frac{\mathsf U}{a\mathsf S}
 \leq1-\frac{a-1}{a(t+2)},\qquad
 (a-1)^2\left(\frac{\mathsf Tt}{a-u}+1\right)
 \leq a(a+t)(t+1).
\end{equation*}
Furthermore,
\begin{equation*}
 (a+1)t+3a-1
 -2\left(1-\frac{a-1}{a(t+2)}\right)a(t+1)
 =\frac{(a-1)(4+t-t^2)}{t+2}>0.
\end{equation*}
Consequently,
\begin{equation*}
 \frac{\mathsf C\mathsf K}{a^2}
 \geq(a+t)\bigl[(a+1)t+3a-1\bigr]
 >2\frac{\mathsf U}{a\mathsf S}(a-1)^2
 \left(\frac{\mathsf Tt}{a-u}+1\right),
\end{equation*}
which proves \eqref{alg-scalar-reduced}. Thus
\eqref{alg-final-scalar-bound} holds, and Lemma~\ref{alg-derivative-estimate}
is proved.
\end{proof}

\section*{Acknowledgements}
This work was supported by the National Key Research and Development Program of China (2021YFA1001800), National Natural Science Foundation of China (No.~12531002) and the Fundamental Research Funds for the Central Universities. The third author was supported by NSFC Grant No. 12601099. The last author was also supported by the China Postdoctoral Science Foundation (Grant No.~2025M783146).

\end{document}